\documentclass[11pt, a4paper, twoside,leqno]{amsart}
\pdfoutput=1					
\usepackage[centering, totalwidth = 380pt, totalheight = 600pt]{geometry}
\usepackage{amssymb, amsmath, amsthm, mathrsfs}
\usepackage{microtype, stmaryrd, url}
\usepackage[shortlabels]{enumitem}
\usepackage[latin1]{inputenc}
\usepackage{xcolor}
\definecolor{darkgreen}{rgb}{0,0.45,0}
\usepackage[colorlinks,citecolor=darkgreen,linkcolor=darkgreen]{hyperref}
\usepackage[british]{babel}
\usepackage{faktor}	
\usepackage{tikz-cd}

\DeclareMathAlphabet{\mathbf}{OT1}{cmr}{b}{n}

\DeclareMathOperator{\colim}{colim}

\newcommand{\cat}[1]{\mathbf{#1}}
\newcommand{\op}{\mathrm{op}}

\renewcommand{\phi}{\varphi}

\newcommand{\E}{{\mathcal E}}

\newcommand{\M}{{\mathcal M}}

\renewcommand{\P}{{\mathcal P}}

\newcommand{\xtor}[1]{\cdl[@1]{{} \ar[r]|-{\object@{|}}^{#1} & {}}}

\newcommand{\To}{\ensuremath{\Rightarrow}}

\makeatletter

\def\hookleftarrowfill@{\arrowfill@\leftarrow\relbar{\relbar\joinrel\rhook}}
\def\twoheadleftarrowfill@{\arrowfill@\twoheadleftarrow\relbar\relbar}
\def\leftbararrowfill@{\arrowdoublefill@{\leftarrow\mkern-5mu}\relbar\mapstochar\relbar\relbar}
\def\Leftbararrowfill@{\arrowdoublefill@{\Leftarrow\mkern-2mu}\Relbar\Mapstochar\Relbar\Relbar}
\def\leftringarrowfill@{\arrowdoublefill@{\leftarrow\mkern-3mu}\relbar{\mkern-3mu\circ\mkern-2mu}\relbar\relbar}
\def\lefttriarrowfill@{\arrowfill@{\mathrel\triangleleft\mkern0.5mu\joinrel\relbar}\relbar\relbar}
\def\Lefttriarrowfill@{\arrowfill@{\mathrel\triangleleft\mkern1mu\joinrel\Relbar}\Relbar\Relbar}

\def\hookrightarrowfill@{\arrowfill@{\lhook\joinrel\relbar}\relbar\rightarrow}
\def\twoheadrightarrowfill@{\arrowfill@\relbar\relbar\twoheadrightarrow}
\def\rightbararrowfill@{\arrowdoublefill@{\relbar\mkern-0.5mu}\relbar\mapstochar\relbar\rightarrow}
\def\Rightbararrowfill@{\arrowdoublefill@{\Relbar\mkern-2mu}\Relbar\Mapstochar\Relbar\Rightarrow}
\def\rightringarrowfill@{\arrowdoublefill@\relbar\relbar{\mkern-2mu\circ\mkern-3mu}\relbar{\mkern-3mu\rightarrow}}
\def\righttriarrowfill@{\arrowfill@\relbar\relbar{\relbar\joinrel\mkern0.5mu\mathrel\triangleright}}
\def\Righttriarrowfill@{\arrowfill@\Relbar\Relbar{\Relbar\joinrel\mkern1mu\mathrel\triangleright}}

\def\leftrightarrowfill@{\arrowfill@\leftarrow\relbar\rightarrow}
\def\mapstofill@{\arrowfill@{\mapstochar\relbar}\relbar\rightarrow}

\renewcommand*\xleftarrow[2][]{\ext@arrow 20{20}0\leftarrowfill@{#1}{#2}}
\providecommand*\xLeftarrow[2][]{\ext@arrow 60{22}0{\Leftarrowfill@}{#1}{#2}}
\providecommand*\xhookleftarrow[2][]{\ext@arrow 10{20}0\hookleftarrowfill@{#1}{#2}}
\providecommand*\xtwoheadleftarrow[2][]{\ext@arrow 60{20}0\twoheadleftarrowfill@{#1}{#2}}
\providecommand*\xleftbararrow[2][]{\ext@arrow 10{22}0\leftbararrowfill@{#1}{#2}}
\providecommand*\xLeftbararrow[2][]{\ext@arrow 50{24}0\Leftbararrowfill@{#1}{#2}}
\providecommand*\xleftringarrow[2][]{\ext@arrow 10{26}0\leftringarrowfill@{#1}{#2}}
\providecommand*\xlefttriarrow[2][]{\ext@arrow 80{24}0\lefttriarrowfill@{#1}{#2}}
\providecommand*\xLefttriarrow[2][]{\ext@arrow 80{24}0\Lefttriarrowfill@{#1}{#2}}

\renewcommand*\xrightarrow[2][]{\ext@arrow 01{20}0\rightarrowfill@{#1}{#2}}
\providecommand*\xRightarrow[2][]{\ext@arrow 04{22}0{\Rightarrowfill@}{#1}{#2}}
\providecommand*\xhookrightarrow[2][]{\ext@arrow 00{20}0\hookrightarrowfill@{#1}{#2}}
\providecommand*\xtwoheadrightarrow[2][]{\ext@arrow 03{20}0\twoheadrightarrowfill@{#1}{#2}}
\providecommand*\xrightbararrow[2][]{\ext@arrow 01{22}0\rightbararrowfill@{#1}{#2}}
\providecommand*\xRightbararrow[2][]{\ext@arrow 04{24}0\Rightbararrowfill@{#1}{#2}}
\providecommand*\xrightringarrow[2][]{\ext@arrow 01{26}0\rightringarrowfill@{#1}{#2}}
\providecommand*\xrighttriarrow[2][]{\ext@arrow 07{24}0\righttriarrowfill@{#1}{#2}}
\providecommand*\xRighttriarrow[2][]{\ext@arrow 07{24}0\Righttriarrowfill@{#1}{#2}}

\providecommand*\xmapsto[2][]{\ext@arrow 01{20}0\mapstofill@{#1}{#2}}
\providecommand*\xleftrightarrow[2][]{\ext@arrow 10{22}0\leftrightarrowfill@{#1}{#2}}
\providecommand*\xLeftrightarrow[2][]{\ext@arrow 10{27}0{\Leftrightarrowfill@}{#1}{#2}}

\makeatother

\numberwithin{equation}{section}

\theoremstyle{plain}
\newtheorem{Thm}{Theorem}
\newtheorem{Prop}[Thm]{Proposition}
\newtheorem{Cor}[Thm]{Corollary}
\newtheorem{Lemma}[Thm]{Lemma}

\theoremstyle{definition}
\newtheorem{Defn}[Thm]{Definition}

\newtheorem{Ex}[Thm]{Example}
\newtheorem{Exs}[Thm]{Examples}
\newtheorem{Rk}[Thm]{Remark}

\newcommand{\Cat}{\cat{Cat}}
\newcommand{\Set}{\cat{Set}}
\newcommand{\CAT}{\cat{CAT}}

\newcommand{\Nat}{\cat{Nat}}
\newcommand{\Sub}{\cat{Sub}}

\newcommand{\msf}[1]{\mathsf{#1}}						
\newcommand{\mbb}[1]{\mathbb{#1}}						
\newcommand{\wt}[1]{\widetilde{#1}}						
\newcommand{\ov}[1]{\overline{#1}}							
\newcommand{\wh}[1]{\widehat{#1}}						
\newcommand{\dom}{\msf{dom}}

\newcommand{\yon}{\mathbf{y}}								
\newcommand{\PSh}{\msf{PSh}}

\newcommand{\Par}{\msf{Par}}
\newcommand{\Total}{\msf{Total}}

\begin{document}
\leftmargini=2em
\title{Cocompletion of restriction categories}
\author{Richard Garner}
\author[Daniel Lin]{Daniel Lin\textsuperscript{$\ast$}}
\address{Department of Mathematics, Macquarie University, North Ryde, NSW 2109, Australia}
\email{richard.garner@mq.edu.au}
\email{daniel.lin@mq.edu.au}

\subjclass[2000]{Primary: 18B99}
\keywords{restriction categories, cocompletion}
\date{\today}

\thanks{\textsuperscript{$\ast$}Corresponding author. 
	The support of a Macquarie University Research Scholarship is gratefully acknowledged.}

\begin{abstract}
	Restriction categories were introduced as a way of generalising the notion of partial map categories. In this paper, we 
	define cocomplete restriction category, and give the free cocompletion of a small restriction category as a suitably
	defined category of restriction presheaves. We also consider the case where our restriction category is locally small.
\end{abstract}
\maketitle

\section{Introduction}
The notion of a partial function is ubiquitous in many areas of mathematics, most notably in computability theory, 
complexity theory, algebraic geometry and algebraic topology. However, such notions of partiality need not be solely 
restricted to sets and partial functions between them, but may also arise in the context of continuous functions
on the open subsets of topological spaces \cite[p.~97]{MR968102}. An early attempt at describing an abstract notion 
of partiality came from Carboni \cite{MR913967}, who considered bicategories with a tensor product and a unique 
cocommutative comonoid structure. However, the first real attempt at axiomatising the general theory came from Di Paola and 
Heller \cite{MR902979}, who introduced the notion of a \emph{dominical category}. Around the same time, Robinson and 
Rosolini \cite{MR968102} gave their own interpretation of this notion of partiality through what they called 
$p$-\emph{categories}, and observed that Di Paolo and Heller's dominical categories were in fact instances of 
$p$-categories.

The common theme between dominical categories and $p$-categories is their reliance on classes of monomorphisms for 
partiality. However, it was shown by Grandis \cite{MR1108477} that it was possible to capture the partiality of maps in the 
form of idempotents on their domains, via the notion of $e$-\emph{cohesive categories}. This same idea was later 
reformulated and studied extensively by Cockett and Lack in their series of three papers on \emph{restriction categories} 
\cite{MR1871071, MR1963657, MR2347616}. Informally, in a restriction category $\cat{X}$, the restriction of a map 
$f\colon A\to B\in\cat{X}$ is an idempotent $\bar{f}\colon A\to A$ which measures the degree of the partiality of $f$. 
In particular, in the category of sets and partial functions, the restriction of a map $f\colon A\to B$ is a partial identity
map on $A$ which has the same domain of definition as $f$.

Since restriction categories are categories with extra structure, it would not be too far-fetched to think that one could give 
a notion of colimits in this restriction setting. As a matter of fact, Cockett and Lack give an explicit description of 
\emph{restriction coproducts} in a restriction category \cite[Lemma 2.1]{MR2347616}. As a necessary first step towards
understanding restriction colimits in general, we consider the notion of a \emph{cocomplete restriction category}, and 
of free restriction cocompletion; indeed this is what we will do in this paper. Future work will include extending this 
notion of restriction cocompletion to join restriction categories, and showing that the manifold completion of a join
restriction category \cite{MR1108477} is a full subcategory of this join restriction cocompletion, whatever that might be. 
Another possibility is to extend this to categories with a restriction tangent structure, and showing that its free
cocompletion also has a restriction tangent structure \cite{MR3192082}.

The starting point for our discussion will be a revision of background material from Cockett and Lack \cite{MR1871071}, 
in section \hyperref[sec2]{2}. In section \hyperref[sec3]{3}, we define cocomplete $\M$-category and cocontinuous 
$\M$-functors. Then using the fact $\M$-categories are \emph{the same} as split restriction categories, we give a 
definition of cocomplete restriction category and cocontinuous restriction functors. We also show that the Cockett-Lack 
embedding exhibits the split restriction category $\Par(\PSh_{\M}(\M\Total(\msf{K}_r(\cat{X}))))$ as the free cocompletion 
of any small restriction category $\cat{X}$.

In section \hyperref[sec4]{4}, we introduce the notion of restriction presheaf on a restriction category $\cat{X}$,
and give an explicit description of the split restriction category of restriction presheaves $\PSh_r(\cat{X})$. 
Finally, we show that this restriction presheaf category $\PSh_r(\cat{X})$ is in fact equivalent to 
$\Par(\PSh_{\M}(\M\Total(\msf{K}_r(\cat{X}))))$, and this in turn gives us an alternate formulation of restriction free 
cocompletion.

Finally in section \hyperref[sec5]{5}, we consider the case where our $\M$-category $\cat{C}$ may not be small, but 
locally small, and give a definition of what it means for an $\M$-category to be locally small. We see that for any locally 
small $\M$-category $\cat{C}$, the $\M$-category of small presheaves $\P_{\M}(\cat{C})$ is not only locally small and 
cocomplete, but is also the free cocompletion of $\cat{C}$. Then as before, it turns out that for any locally small restriction 
category $\cat{X}$, the Cockett-Lack embedding exhibits the restriction category 
$\Par(\P_{\M}(\M\Total(\msf{K}_r(\cat{X}))))$ as its free cocompletion. Also, just as small presheaves are defined to 
be a small colimit of representables, we define small restriction presheaves analogously.

\section{Restriction category preliminaries}\label{sec2}
	\subsection{Restriction categories}
	In this section, we recall the definition of a restriction category and basic lemmas from \cite{MR1871071}. 
	We recall there is a $2$-category of restriction categories called $\cat{rCat}$, and that $\cat{rCat}$ has an important
	sub-$2$-category $\cat{rCat}_s$ of split restriction categories. The reason for its importance is due to 
	\cite[Theorem 3.4]{MR1871071}, which says there is an equivalence between $\cat{rCat}_s$ and the $2$-category $\M\Cat$ of
	$\M$-categories (or categories with a stable system of monics). A consequence of this theorem is that it allows us to work
	with $\M$-categories, which are not much different to ordinary categories, and transfer any results obtained across to restriction 
	categories. We will be referring frequently to this equivalence between $\cat{rCat}_s$ and $\M\Cat$ in later sections.
	
	\begin{Defn}
		A \emph{restriction category} is a category $\cat{X}$ together with assignations
			$$ \cat{X}(A,B) \to \cat{X}(A,A), \quad f \mapsto \bar{f} $$
		where $\bar{f}$ satisfies the following conditions:
			\begin{enumerate}[leftmargin=1.5cm,label=(R\arabic*)]
				\item $f \circ \bar{f} = f$;
				\item $\bar{g} \circ \bar{f} = \bar{f} \circ \bar{g}$ for $f \colon A \to B$, $g\colon A \to C$;
				\item $\ov{g \circ \bar{f}} = \bar{g} \circ \bar{f}$ for $f\colon A \to B$, $g\colon A \to C$;
				\item $\bar{h} \circ f = f \circ \ov{h \circ f}$ for $f\colon A \to B$, $h\colon B \to C$.
			\end{enumerate}
		The assignations $f\mapsto \bar{f}$ are called the \emph{restriction structure} on $\cat{X}$, and we call
		$\bar{f}$ the \emph{restriction} of $f$.
	\end{Defn}
	
	\begin{Exs}
		\begin{enumerate}[label=(\arabic*)]
			\item The category of sets and partial functions $\Set_p$ is a restriction category, where the restriction on each partial function
						$f\colon A \to B$ is given by
							$$ \bar{f}(a) = \begin{cases} a & \text{if $f$ is defined at $a\in A$;} \\
						\text{undefined} & \text{otherwise.} \end{cases} $$
			\item Consider the set of natural numbers $\mbb{N}$ as a monoid whose composition is given by $n\circ m=\max(m,n)$. Then 
						$\mbb{N}$ maybe given two restriction structures; the first by $\bar{n}=n$, and the second by
							$$ \bar{n} = \begin{cases} n & \text{$n=0$ or $n$ odd;} \\ n-1 & \text{otherwise.} \end{cases} $$
		\end{enumerate}
	\end{Exs}
	
	The restriction $\bar{f}$ of any map $f$ in a restriction category satisfies the following basic properties (see 
	\cite[pp.~227,230]{MR1871071} for details).
	\begin{Lemma}
		Let $\cat{X}$ be a restriction category, and let $f\colon A \to B$ and $g\colon B\to C$ be morphisms in $\cat{X}$. Then
		\begin{enumerate}[leftmargin=1cm,label=(\arabic*)]
			\item $\bar{f}$ is idempotent;
			\item $\bar{f} \circ \ov{gf} = \ov{gf}$;
			\item $\ov{\bar{g} f} = \ov{gf}$;
			\item $\bar{f} = \bar{\bar{f}}$;
			\item if $f$ is a monomorphism, then $\bar{f}=1$;
			\item $\cat{X}(A,B)$ has a partial order given by $f \le f'$ if and only if $f=f' \circ \bar{f}$.
		\end{enumerate}
	\end{Lemma}
	A map $f \in \cat{X}$ is called a \emph{restriction idempotent} if $f = \bar{f}$, and is \emph{total} if $\bar{f} = 1$. If $f\colon A\to B$
	and $g\colon B\to C$ are total maps in a restriction category, then $gf$ is also total since
		$$ \ov{gf}= \ov{\bar{g}f} = \ov{f} = 1. $$
	Therefore, as identities are total, the objects and total maps of any restriction category $\cat{X}$ form a subcategory 
	$\msf{Total}(\cat{X})$.
	
	\begin{Defn}
		A functor $F\colon\cat{X}\to\cat{Y}$ between restriction categories is called a \emph{restriction functor} if 
		$F(\bar{f})=\ov{F(f)}$ for all maps $f\in\cat{X}$, and a natural transformation $\alpha\colon F\To G$ is a \emph{restriction
		transformation} if its components are total. We denote by $\cat{rCat}$ the $2$-category of restriction categories (objects), 
		restriction functors ($1$-cells) and restriction transformations ($2$-cells).
	\end{Defn}

	\subsection{Split restriction categories}
	There is an important full sub-$2$-category $\cat{rCat}_s$ of $\cat{rCat}$, the objects of which are restriction categories whose 
	restriction idempotents split. Recall that a restriction idempotent $\bar{f}$ \emph{splits} if there exist maps $m$ and $r$ such that 
	$mr = \bar{f}$ and $rm=1$. We call such maps $m$, \emph{restriction monics}. 
		
	The inclusion $\cat{rCat}_s \hookrightarrow \cat{rCat}$ has a left biadjoint $\msf{K}_r$, which on objects takes restriction 
	categories $\cat{X}$ to split restriction categories $\msf{K}_r(\cat{X})$ \cite[p.~242]{MR1871071} with the following data:
		\begin{description}[labelindent=0.5cm,leftmargin=0.5cm]
			\item [Objects] Pairs $(A,e)$, where $A$ is an object of $\cat{X}$ and $e\colon A\to A$ is a restriction idempotent on $A$;
			\item [Morphisms] Morphisms $f\colon (A,e) \to (A',e')$ are morphisms $f\colon A\to A'$ in $\cat{X}$ satisfying the 
						condition $e'fe=f$;
			\item [Restriction] Restriction on $f$ is given by $\bar{f}$. 
		\end{description}
	The unit at $\cat{X}$ of this biadjunction, $J\colon \cat{X} \to \msf{K}_r(\cat{X})$, takes an object $A$ to $(A,1_A)$ and a morphism 
	$f\colon A\to A'$ to $f \colon(A,1_A) \to (A',1_{A'})$ in $\msf{K}_r(\cat{X})$. As alluded to earlier, this $2$-category of split restriction 
	categories $\cat{rCat}_s$ is equivalent to a $2$-category called $\M\Cat$, the objects of which form the basis for our discussion 
	in the next section.
	
	\subsection{\texorpdfstring{$\M$}{M}-categories and partial map categories}
	A \emph{stable system of monics} $\M_{\cat{C}}$ in a category $\cat{C}$ is a collection of monics in $\cat{C}$ which includes all 
	isomorphisms, is closed under composition, and the pullback of any $m\in\M_{\cat{C}}$ along arbitrary maps in $\cat{C}$ exists
	and is in $\M_{\cat{C}}$. An $\M$-category is then a category $\cat{C}$ together with a stable system of monics 
	$\M_{\cat{C}} \subseteq \cat{C}$, which we write as a pair $(\cat{C},\M_{\cat{C}})$ \cite[p.~245]{MR1871071}. (Where the 
	meaning is clear, we shall dispense with the notation $(\cat{C},\M_{\cat{C}})$ and simply write $\cat{C}$).
	
	If $\cat{C}$ and $\cat{D}$ are $\M$-categories, a functor $F$ between them is called an \emph{$\M$-functor} if 
	$m\in\M_{\cat{C}}$ implies $Fm\in\M_{\cat{D}}$, and $F$ preserves pullbacks of monics in $\M_{\cat{C}}$. Further, if 
	$F,G \colon \cat{C} \to \cat{D}$ are $\M$-functors, a natural transformation between them is called \emph{$\M$-cartesian} if 
	the naturality square is a pullback for all $m \in \M_{\cat{C}}$ \cite[p.~247]{MR1871071}. We denote by $\M\Cat$ the 
	$2$-category of $\M$-categories (objects), $\M$-functors ($1$-cells) and $\M$-cartesian natural transformations ($2$-cells).
	
	Now associated with any $\M$-category $\cat{C}$ is the split restriction category $\Par(\cat{C})$, called the \emph{category of partial
	maps in $\cat{C}$}. It has the same objects as $\cat{C}$, and morphisms from $X \to Y$ in $\Par(\cat{C})$ are spans 
		$$ (m,f) = X \xleftarrow{m\,\in\,\M_{\cat{C}}} Z \xrightarrow{f} Y $$ 
	identified up to some equivalence class. More precisely, $(m,f) \sim (n,g)$ if and only if there exists an 
	isomorphism $\phi$ such that $m\phi = n$ and $f\phi=g$. Composition in this category is by pullback, the identity is given by $(1,1)$ 
	and the restriction of $(m,f)$ is $(m,m)$ \cite[pp.~246,247]{MR1871071}.
	
	There is also a $2$-functor $\Par\colon \M\Cat \to \cat{rCat}_s$ which on objects, takes $\M$-categories $\cat{C}$ to 
	split restriction categories $\Par(\cat{C})$. If $F \colon \cat{C} \to \cat{D}$ is an $\M$-functor, then $\Par(F)$ takes objects 
	$A \in \Par(\cat{C})$ to $FA$ and morphisms $(m,f)$ to $(Fm,Ff)$. Also, if $\alpha \colon F \Rightarrow G$ is $\M$-cartesian, 
	then $\Par(\alpha)$ is defined componentwise by $\Par(\alpha)_A = (1_{FA},\alpha_A)$.
	
	\begin{Thm}
		The $2$-functor $\Par \colon\M\Cat \to \cat{rCat}_s$ is an equivalence of $2$-categories.
	\end{Thm}
	\begin{proof}
		We give a quick sketch of the proof. For full details, see \cite[Theorem 3.4]{MR1871071}. We know that $\Par$ is a $2$-functor 
		from $\M\cat{Cat}$ to $\cat{rCat}_s$. Likewise, there is a $2$-functor $\M\Total \colon \cat{rCat}_s \to \M\cat{Cat}$, taking split 
		restriction categories $\cat{X}$ to $\M$-categories $(\Total(\cat{X}),\M_{\Total(\cat{X})})$, where $\M_{\Total(\cat{X})}$ consists of 
		the restriction monics in $\cat{X}$. (Recall that $\M\Total(\cat{X})$ really is an $\M$-category \cite[Proposition 3.3]{MR1871071}). 
		
		The pair $\Par$ and $\M\msf{Total}$ are then part of a $2$-equivalence, with the unit at $\cat{X}$,
		$\Phi_{\cat{X}} \colon \cat{X} \to \Par(\M\Total(\cat{X}))$, given by $\Phi_{\cat{X}}(A) = A$ on objects and by
		$\Phi_{\cat{X}}(f) = (m,fm)$ on arrows (where $\bar{f}=mr$ and $rm=1$). On the other hand, the counit at $\cat{C}$ is defined by
		$\Psi_{\cat{C}}(A) = A$ on objects and $\Psi_{\cat{C}}(1,f) = f$ on morphisms.
	\end{proof}
	
\section{Cocompletion of restriction categories}\label{sec3}
For any small category $\cat{C}$, we may characterise the category of presheaves $\PSh(\cat{C})$ as the 
\emph{free cocompletion} of $\cat{C}$. That is, for any small-cocomplete category $\E$, the following is an equivalence of categories:
	$$ (-)\circ\yon\colon \cat{Cocomp}(\PSh(\cat{C}),\E) \to \Cat(\cat{C},\E) $$
where $\yon$ is the Yoneda embedding, $\Cat$ is the $2$-category of small categories and $\cat{Cocomp}$ is the $2$-category of 
small-cocomplete categories and cocontinuous functors. (For the rest of this paper, we shall take \emph{cocomplete} to mean 
\emph{small-cocomplete}, and \emph{colimits} to mean \emph{small colimits} unless otherwise indicated). However, it is not 
immediately obvious that there is an analogous notion of cocompletion for any small restriction category $\cat{X}$. 
Nonetheless, a clue is given to us in light of the $2$-equivalence between $\M\Cat$ and $\cat{rCat}_s$. That is, it might be helpful 
to first define a notion of cocomplete $\M$-category, and study the free cocompletion of small $\M$-categories.

In this section, we recall the $\M$-category of presheaves $\PSh_{\M}(\cat{C})$ for any small $\M$-category $\cat{C}$ and give 
a definition of cocomplete $\M$-category and cocontinuous $\M$-functor. (As it turns out, this $\M$-category of presheaves,
$\PSh_{\M}(\cat{C})$ will be the free cocompletion of any small $\M$-category $\cat{C}$). Then using the $2$-equivalence between 
$\M\Cat$ and $\cat{rCat}_s$, we define cocomplete restriction categories and cocontinuous restriction functors. This in turn provides 
a candidate for free restriction cocompletion, namely the split restriction category $\Par(\PSh_{\M}(\M\Total(\msf{K}_r(\cat{X}))))$ 
described by Cockett and Lack \cite{MR1871071}.

\subsection{An \texorpdfstring{$\M$}{M}-category of presheaves}
For any small $\M$-category $\cat{C}$, there are various ways of constructing an $\M$-category of presheaves on $\cat{C}$.
One way is the following, and we denote the $\M$-category arising in this way by 
$\PSh_{\M}(\cat{C}) = (\PSh(\cat{C}),\M_{\PSh(\cat{C})})$. We say a map $\mu \colon P \To Q$ is an 
\emph{$\M_{\PSh(\cat{C})}$-map} if for all $\gamma \colon \yon D \To Q$, there exists an $m \in \M_{\cat{C}}$ making the 
following a pullback square:
	\begin{center}\begin{tikzcd}
		\yon C \arrow{r} \arrow[swap]{d}{\yon m} & P \arrow{d}{\mu} \\
		\yon D \arrow[swap]{r}{\gamma} & Q
	\end{tikzcd}\end{center}
where $\yon \colon \cat{C} \to \PSh(\cat{C})$ is the usual Yoneda embedding \cite[p.~252]{MR1871071}. Observe that under 
this construction, the Yoneda embedding is an $\M$-functor $\yon\colon \cat{C} \to \PSh_{\M}(\cat{C})$.

\subsection{Cocomplete \texorpdfstring{$\M$}{M}-categories}
It is well known that for any small $\M$-category $\cat{C}$, the Yoneda embedding $\yon\colon\cat{C}\to\PSh(\cat{C})$ 
exhibits $\PSh(\cat{C})$ as the free cocompletion of $\cat{C}$. Therefore it is natural to ask whether for any small $\M$-category
$\cat{C}$, the Yoneda embedding $\yon\colon\cat{C}\to\PSh_{\M}(\cat{C})$ likewise exhibits $\PSh_{\M}(\cat{C})$ as the free 
cocompletion of $\cat{C}$. First we need to give a definition of cocomplete $\M$-category and cocontinuous $\M$-functor.

\begin{Defn}
	An $\M$-category $(\cat{C},\M_{\cat{C}})$ is cocomplete if $\cat{C}$ is itself cocomplete and its inclusion into 
	$\Par(\cat{C})$ preserves colimits. An $\M$-functor $F \colon (\cat{C},\M_{\cat{C}}) \to (\cat{D},\M_{\cat{D}})$ between 
	$\M$-categories is cocontinuous if the underlying functor $F\colon \cat{C}\to\cat{D}$ is cocontinuous. We denote by 
	$\M\cat{Cocomp}$ the $2$-category of cocomplete $\M$-categories, cocontinuous $\M$-functors and $\M$-cartesian 
	natural transformations.
\end{Defn}

\begin{Ex}\label{SetM}
	Let $\Set$ denote the category of all small sets, and consider the $\M$-category $(\Set,\M_{\Set})$, where $\M_{\Set}$
	are all the injective functions. Then, as $\Set$ is cocomplete and $\Set \hookrightarrow \Par(\Set,\M_{\Set}) = \Set_p$ has a 
	right adjoint, $(\Set,\M_{\Set})$ is a cocomplete $\M$-category.
\end{Ex}

As a matter of fact, there are whole classes of examples of cocomplete $\M$-categories. Before we give their construction,
it will be helpful to define what we mean by an $\M$-subobject.

\begin{Defn}{($\M$-subobjects)}
	Let $\cat{C}$ be an $\M$-category and $D$ an object in $\cat{C}$. Then an $\M$-subobject is an isomorphism
	class of $\M_{\cat{C}}$-maps with codomain $D$. That is, if $m\colon C \to D$ and $m'\colon C' \to D$ are both 
	in $\M_{\cat{C}}$, then $m$ and $m'$ represent the same subobject of $D$ if there exists an isomorphism 
	$\varphi\colon C \to C'$ such that $m = m' \varphi$. We shall use the notation $\Sub_{\M_{\cat{C}}}(D)$ to denote
	the set of subobjects of $D$ in the $\M$-category $\cat{C}$.
\end{Defn}

It will be useful to observe the following lemma in relation to $\M$-subobjects of representables in the $\M$-category 
$\PSh_{\M}(\cat{C})$.

\begin{Lemma}\label{MSubRep}
	Let $\cat{C}$ be an $\M$-category. Then there exists an isomorphism as follows:
		$$ \Sub_{\M_{\PSh(\cat{C})}}(\yon C) \cong \Sub_{\M_{\cat{C}}}(C). $$
\end{Lemma}

\begin{proof}
	Define a function $\varphi\colon \Sub_{\M_{\cat{C}}}(C) \to \Sub_{\M_{\PSh(\cat{C})}}(\yon C)$, which takes
	an $\M$-subobject $m\colon D\to C$ to $\yon m\colon \yon D \to \yon C$, a map in $\M_{\PSh(\cat{C})}$.
	To define its inverse, consider the function $\psi\colon \Sub_{\M_{\PSh(\cat{C})}}(\yon C) \to \Sub_{\M_{\cat{C}}}(C)$
	which takes an $\M$-subobject of $\yon C$, $\mu\colon P \to \yon C$, to the unique subobject $n\colon A \to C$ making 
	the diagram on the left a pullback:
		\begin{center}
			\begin{tikzcd}
				\yon A \arrow{r}{\alpha} \arrow[swap]{d}{\yon n} & P \arrow{d}{\mu} \\
				\yon C \arrow[swap]{r}{1_{\yon C}} & \yon C
			\end{tikzcd} \hspace{2cm}
			\begin{tikzcd}[row sep=tiny,column sep=tiny]
				P \arrow[bend left=30]{drrr}{1} \arrow[swap,bend right=30]{dddr}{\mu} \arrow[dashed]{dr}{\beta} & & & \\
				& \yon A \arrow{rr}{\alpha} \arrow[swap]{dd}{\yon n} & & P \arrow{dd}{\mu} \\
				& & \phantom{m} & \\
				& \yon C \arrow[swap]{rr}{1_{\yon C}} & & \yon C
			\end{tikzcd}
		\end{center}
	Clearly $\psi\circ\varphi = 1$. To see that $\varphi\circ\psi = 1$, consider the previous diagram on the right. By definition, 
	there exists a unique map $\beta$ such that $\yon n \circ \beta = \mu$ and $\alpha\beta=1_P$. But 
	$\yon n = \yon n \circ \beta \circ \alpha$ and $\yon n$ is monic, which means $\beta\alpha=1$. Therefore, 
	$\mu$ and $\yon n$ belong to the same isomorphism class of $\M$-subobjects of $\yon C$. Hence 
	$\varphi\circ\psi=1$, and so $\Sub_{\M_{\PSh(\cat{C})}}(\yon C) \cong \Sub_{\M_{\cat{C}}}(C)$.
\end{proof}

Now consider an $\M$-category $(\E,\M_{\E})$, where $\M_{\E}$ is a stable system of monics and $\E$ is a 
cocomplete category with a terminal object $1$ and a generic $\M$-subobject $\tau\colon 1\to\Sigma$. By a generic 
$\M$-subobject (or an $\M$-subobject classifier), we mean an object $\Sigma\in\E$ and an $\M_{\E}$-map 
$\tau\colon 1\to\Sigma$ such that for any $\M_{\E}$-map $m\colon A \to B$, there exists a unique map 
$\tilde{m}\colon B\to\Sigma$ making the following square a pullback:
	\begin{center}\begin{tikzcd}
		A \arrow[swap]{d}{m} \arrow{r} & 1 \arrow{d}{\tau} \\
		B \arrow[swap]{r}{\tilde{m}} & \Sigma
	\end{tikzcd}\end{center}
Suppose the induced pullback functor $\tau^*\colon \E / \Sigma \to\E$ has a right adjoint $\Pi_{\tau}$. Then by
an analogous argument in topos theory \cite[Proposition~2.4.7]{MR1953060}, $\E$ has a \emph{partial map classifier}
for every object $C\in\E$, and this in turn implies that the inclusion $\E \hookrightarrow \Par(\E,\M_{\E})$ has a right 
adjoint \cite[p.~65]{MR1963657}, and so $\M$-categories of this kind are cocomplete.

\begin{Exs}\label{CocompMCat}
	\begin{enumerate}[label=(\arabic*)]
		\item Let $\E$ be any cocomplete elementary topos, and let $\M_{\E}$ be all the monics in $\E$. Then $(\E,\M_{\E})$ 
				is a cocomplete $\M$-category since $\E$ is locally cartesian closed and has a generic subobject.
		\item If $\E$ is any cocomplete quasitopos and $\M_{\E}$ are all the regular monics in $\E$, then
				$(\E,\M_{\E})$ is also a cocomplete $\M$-category as it is locally cartesian closed and has an object which
				classifies all the regular monics in $\E$.
		\item 
				We know the presheaf category on any small category $\cat{C}$ is cocomplete and locally cartesian closed.
				So consider the $\M$-category $\PSh_{\M}(\cat{C})$. If an $\M$-subobject classifier were to exist,
				then by Yoneda, we would have
					$$ \Sigma (C) \cong \PSh(\cat{C})(\yon C,\Sigma) \cong \Sub_{\M_{\PSh(\cat{C})}}(\yon C). $$
				But because $\Sub_{\M_{\PSh(\cat{C})}}(\yon C) \cong \Sub_{\M_{\cat{C}}}(C)$ (Lemma \ref{MSubRep}), 
				define $\Sigma$ to take objects $C\in \cat{C}$ to the set of $\M$-subobjects of $C$, and maps 
				$f\colon D \to C$ in $\cat{C}$ to $f^*$, the change-of-base functor (by pullback along $f$). Finally, define the 
				map $\tau\colon 1\to\Sigma$ componentwise at $C\in\cat{C}$ by taking the singleton to the largest 
				$\M$-subobject of $C$, the identity on $C$.
				
				It is then not difficult to check that this map $\tau\colon 1\to\Sigma$ is in $\M_{\PSh(\cat{C})}$, and also 
				classifies all $\M_{\PSh(\cat{C})}$-maps. Hence, $\PSh_{\M}(\cat{C})$ is a cocomplete $\M$-category.
	\end{enumerate}
\end{Exs}

The following proposition gives an alternative characterisation of the inclusion $\cat{C} \hookrightarrow \Par(\cat{C})$ being 
cocontinuous for a cocomplete category $\cat{C}$.

\begin{Prop}\label{MCocompAlt}
	Suppose $(\cat{C},\M_{\cat{C}})$ is an $\M$-category, and $\cat{C}$ is cocomplete. Then the following statements 
	are equivalent:
	\begin{enumerate}[label=(\arabic*)]
		\item The inclusion $\cat{C} \hookrightarrow \Par(\cat{C})$ preserves colimits;
		\item The following conditions hold:
				\begin{enumerate}[label=(\alph*)]
					\item If $\{m_i\colon A_i \to B_i \}_{i\in I}$ is a family of maps in $\M_{\cat{C}}$ indexed by a small set $I$, 
						then their coproduct $\sum_{i\in I} m_i$ is in $\M_{\cat{C}}$ and the following squares are pullbacks
						for every $i\in I$:
						\begin{center}\begin{tikzcd}
							A_i \arrow{r}{\imath_{A_i}} \arrow[swap]{d}{m_i} & \sum_{i\in I} A_i \arrow{d}{\sum_{i\in I} m_i} \\
							B_i \arrow[swap]{r}{\imath_{B_i}} & \sum_{i\in I} B_i
						\end{tikzcd}\end{center}
					\item Suppose $m\in\M_{\cat{C}}$ and the pullback of $m$ along two maps $f,g \in \cat{C}$ is the same
						 map $h$. If $f',g'$ are the pullbacks of $f,g$ along $m$, and $c,c'$ are the coequalisers of $f,g$ and $f',g'$
						 respectively, then the unique $n$ making the right square commute is in $\M_{\cat{C}}$ and also
						 makes the right square a pullback:
						\begin{center}\begin{tikzcd}
							\bullet \arrow[start anchor=20,end anchor=160]{r}{f'} 
										\arrow[start anchor=base east,end anchor=base west,swap]{r}{g'} \arrow[swap]{d}{h} 
									& \bullet \arrow{r}{c'} \arrow{d}{m} & \bullet \arrow{d}{n} \\
							\bullet \arrow[start anchor=20,end anchor=160]{r}{f} 
										\arrow[start anchor=base east,end anchor=base west,swap]{r}{g} & \bullet \arrow[swap]{r}{c} & \bullet
						\end{tikzcd}\end{center}
					\item Colimits are stable under pullback along $\M_{\cat{C}}$-maps.
				\end{enumerate}
	\end{enumerate}
\end{Prop}

\begin{proof}
	For the proof of $(1) \implies (2)$, we will be using Lemma \ref{MCatLem2} and Corollary \ref{MCatLem3Cor} (both
	to be proven later).
	
	$(1) \implies (2a)$ Let $I$ be a small discrete category, and let $H,K \colon I \to\cat{C}$ be functors taking objects
	$i\in I$ to $A_i$ and $B_i$ respectively. Let $\alpha\colon H\To K$ be a natural transformation whose component at
	$i$ is given by $m_i\colon A_i\to B_i$, and observe that all naturality squares are trivially pullbacks. Then by 
	Lemma \ref{MCatLem2}, the sum $\sum_{i\in I}m_i$ is in $\M_{\cat{C}}$ and for every $i\in I$, the coproduct 
	coprojection squares are pullbacks.
	
	$(1) \implies (2b)$ Take $I$ to be the category with two objects and a pair of parallel maps between them and
	apply Lemma \ref{MCatLem2}.
	
	$(1) \implies (2c)$ See Corollary \ref{MCatLem3Cor}.
	
	$(2) \implies (1)$ Conversely, to show that the inclusion $\cat{C}\hookrightarrow\Par(\cat{C})$ is 
	cocontinuous, it is enough to show that it preserves all small coproducts and coequalisers. 
	
	So suppose $c$ is a coequaliser of $f$ and $g$ in $\cat{C}$.
		\begin{center}\begin{tikzcd}
			\bullet \arrow[start anchor=20,end anchor=160]{r}{f} 
						\arrow[start anchor=base east,end anchor=base west,swap]{r}{g} & \bullet \arrow{r}{c} & \bullet
		\end{tikzcd}\end{center}
	To show the inclusion preserves this coequaliser, we need to show that for any map $(m,k)$ such that 
	$(m,k)(1,f) = (m,k)(1,g)$, there is a unique map $(n,q)$ making the following diagram commute:
		\begin{center}\begin{tikzcd}
			\bullet \arrow[start anchor=20,end anchor=160]{r}{(1,f)} 
						\arrow[start anchor=base east,end anchor=base west,swap]{r}{(1,g)} 
				& \bullet \arrow{r}{(1,c)} \arrow[swap]{dr}{(m,k)} & \bullet \arrow[dashed]{d}{(n,q)} \\
				& & \bullet
		\end{tikzcd}\end{center}
	Now the condition $(m,k)(1,f) = (m,k)(1,g)$ is precisely the condition that the pullbacks of $m$ along $f$ and $g$ are 
	the same map $h$,
		\begin{center}\begin{tikzcd}
			\bullet \arrow[start anchor=20,end anchor=160]{r}{f'} 
						\arrow[start anchor=base east,end anchor=base west,swap]{r}{g'} \arrow[swap]{d}{h} 
					& \bullet \arrow{d}{m} \\
			\bullet \arrow[start anchor=20,end anchor=160]{r}{f} 
						\arrow[start anchor=base east,end anchor=base west,swap]{r}{g} & \bullet
		\end{tikzcd}\end{center}
	and that $kf'=kg'$. Taking $c'$ to be the coequaliser of $f'$ and $g'$, our assumption then implies there
	is a unique map $n\in \M_{\cat{C}}$ making the following diagram a pullback:
		\begin{center}\begin{tikzcd}
			\bullet \arrow{r}{c'} \arrow[swap]{d}{m} & \bullet \arrow{d}{n} \\
			\bullet \arrow[swap]{r}{c} & \bullet
		\end{tikzcd}\end{center}
	Since $c'$ is the coequaliser of $f'$ and $g'$ and $kf'=kg'$, there exists a unique map $q$ such that
	$c'q=k$. This gives a map $(n,q) \in\Par(\cat{C})$ such that $(n,q)(1,c)=(m,k)$. To see it must be unique,
	suppose $(n',q')$ also satisfies the condition $(n',q')(1,c)=(m,k)$. By assumption, as colimits are stable under 
	pullback along $\M_{\cat{C}}$-maps, the pullback of $c$ along $n'$ must be a coequaliser of 
	$f'$ and $g'$, say $c''$.
	\begin{center}\begin{tikzcd}[row sep=small,column sep=small]
		\bullet \arrow{rrr}{c''} \arrow[swap]{ddd}{m} \arrow{drr}{c'} & & & \bullet \arrow{ddd}{n'} \\
		& & \bullet \arrow{ur}{\varphi} \arrow{ddr}{n} & \\
		& & & \\
		\bullet \arrow[swap]{rrr}{c} & & & \bullet 
	\end{tikzcd}\end{center}
	Now as coequalisers are unique up to isomorphism, there is an isomoprhism $\varphi$ such that 
	$c'' = \varphi c'$. But the fact
		$$ n' \varphi c' = n' c'' = cm = n c' $$
	implies $n'\varphi = n$ as $c'$ is an epimorphism. In other words, $n$ and $n'$ 
	must be the same $\M$-subobject. Similarly, $q=q'\varphi$, which means $(n,q)=(n',q')$.
	
	Next, suppose $\sum_{i\in I}B_i$ is a small coproduct in $\cat{C}$, with coproduct coprojections 
	$(\imath_{B_i}\colon B_i \to \sum_{i\in I} B_i)_{i\in I}$. Then $\sum_{i\in I}B_i$ will 
	be a small coproduct in $\Par(\cat{C})$ if for any object $D\in\Par(\cat{C})$ and family of maps 
	$\big( (m_i,f_i) \colon B_i \to D \big)_{i\in I}$, there exists a unique map $(\mu,\gamma)\colon \sum_{i\in I}B_i \to D$ 
	making the following diagram commute for every $i\in I$:
		\begin{center}\begin{tikzcd}
			B_i \arrow{r}{(1,\imath_{B_i})} \arrow[swap]{dr}{(m_i,f_i)} & \sum_{i\in I} B_i \arrow{d}{(\mu,\gamma)} \\
			& D
		\end{tikzcd}\end{center}
	By assumption, $\sum_{i\in I}m_i$ is in $\M_{\cat{C}}$, and so the map 
	$\left(\sum_{i\in I}m_i, f \right)\colon \sum_{i\in I}B_i \to D$ is well-defined, where $f$ is the unique map
	$\sum_{i\in I} \dom(f_i) \to D$ induced by the universal property of the coproduct coprojections and
	the family of maps $\{ f_i\}_{i\in I}$. Since the coproduct coprojection squares are pullbacks, taking
	$\mu = \sum_{i\in I}m_i$ and $\gamma = f$ certainly makes the above diagram commute, and the uniqueness of 
	$(\mu,\gamma)$ follows by an analogous argument to the case of coequalisers by the stability of colimits under 
	pullback. Therefore, if $\sum_{i\in I} B_i$ is a small coproduct in $\cat{C}$, it is also a small coproduct in $\Par(\cat{C})$.
	
	Therefore, since the inclusion $\cat{C} \hookrightarrow \Par(\cat{C})$ preserves all small coproducts and all
	coequalisers, it preserves all small colimits.
\end{proof}

\begin{Rk}
	There is yet another formulation for the condition that the inclusion $\cat{C} \hookrightarrow \Par(\cat{C})$
	preserves all small colimits. That is, the inclusion is cocontinuous if and only if the functor 
	$\Sub_{\M_{\cat{C}}} \colon \cat{C}^{\op} \to \Set$, which on objects takes $C$ to the set of 
	$\M$-subobjects of $C$, is continuous, and moreover, colimits are stable under pullback
	along $\M_{\cat{C}}$-maps. The proof of this result is similar to the proof of Lemma \ref{MCocompAlt}.
	
	Also, by conditions (2a) and (2c), observe that cocomplete $\M$-categories must be $\M$-extensive, meaning that 
	for every $i\in I$ (with $I$ small), if the following square is commutative with the bottom row being coproduct 
	injections and $m,m_i \in \M$ (for all $i\in I$), then the top row must be a coproduct diagram if and only if each 
	square is a pullback:
	\begin{center}\begin{tikzcd}
		A_i \arrow{r} \arrow[swap]{d}{m_i} & Z \arrow{d}{m} \\
		B_i \arrow{r}{\imath_{B_i}} & \sum_{i\in I} B_i
	\end{tikzcd}\end{center}
\end{Rk}

In light of the previous proposition, we give an example of an $\M$-category which is not cocomplete.

\begin{Ex} 
	Consider the $\M$-category $(\cat{Ab},\M_{\cat{Ab}})$ of small abelian groups and all monomorphisms in $\cat{Ab}$. 
	Denote the trivial group by $0$ and the group of integers by $\mbb{Z}$. The coproduct of $\mbb{Z}$ with itself is just 
	the direct sum $\mbb{Z} \oplus\mbb{Z}$, along with coprojections $\imath_1\colon\mbb{Z}\to\mbb{Z}\oplus\mbb{Z}$ and 
	$\imath_2\colon\mbb{Z}\to\mbb{Z}\oplus\mbb{Z}$ sending $n$ to $(n,0)$ and $(0,n)$ respectively. Let
	$\Delta\colon\mbb{Z}\to\mbb{Z}\oplus\mbb{Z}$ denote the diagonal map, which is clearly a monomorphism and
	hence lies in $\M_{\cat{Ab}}$. Now a pullback of $\Delta$ along $\imath_1$ is the unique map $0\to\mbb{Z}$, and
	similarly for $\imath_2$. This gives the following diagram, where both squares are pullbacks:
	\begin{center}\begin{tikzcd}
		0 \arrow{r} \arrow{d} & \mbb{Z} \arrow{d}{\Delta} & 0 \arrow{l} \arrow{d} \\
		\mbb{Z} \arrow[swap]{r}{\imath_1} & \mbb{Z}\oplus\mbb{Z} & \mbb{Z} \arrow{l}{\imath_2}
	\end{tikzcd}\end{center}
	However, the top row is certainly not a coproduct diagram in $\cat{Ab}$. Therefore, $(\cat{Ab},\M_{\cat{Ab}})$ is not
	$\M$-extensive, and hence by Proposition \ref{MCocompAlt}, is not a cocomplete $\M$-category.
\end{Ex}

\subsection{Cocompletion of \texorpdfstring{$\M$}{M}-categories}
Our goal is to show for any small $\M$-category $\cat{C}$ and cocomplete $\M$-category $\cat{D}$, the following is an equivalence:
	$$ (-)\circ\yon\colon \M\cat{Cocomp}(\PSh_{\M}(\cat{C}),\cat{D}) \to \M\Cat(\cat{C},\cat{D}). $$
To do so will require the next four lemmas.
	\begin{Lemma}\label{MCatLem0}
		Let $\cat{C}$ be an $\M$-category and let $m\in\M_{\cat{C}}$. Then the following is a pullback square
			\begin{center}\begin{tikzcd}
				A \arrow{r}{g} \arrow[swap]{d}{n} & B \arrow{d}{m} \\
				C \arrow[swap]{r}{f} & D
			\end{tikzcd}\end{center}
		if and only if the following diagram commutes in $\Par(\cat{C})$:
			\begin{center}\begin{tikzcd}
				C \arrow{r}{(1,f)} \arrow[swap]{d}{(n,1)} & D \arrow{d}{(m,1)} \\
				A \arrow[swap]{r}{(1,g)} & B
			\end{tikzcd}\end{center}
	\end{Lemma}
	\begin{proof}
		Diagram chase.
	\end{proof}
	
	\begin{Lemma}\label{MCatLem1}
		Let $\cat{X}$ be a restriction category and $\cat{I}$ any small category. Suppose given $L \colon \cat{I}\to\cat{X}$ and a colimiting
		cocone $p_I \colon LI \to \colim L$ whose colimit coprojections are total. If $\epsilon \colon L \To L$ is a natural
		transformation such that each component is a restriction idempotent, then $\colim \epsilon$ is also a restriction idempotent.
			\begin{center}\begin{tikzcd}
				LI \arrow{r}{p_I} \arrow[swap]{d}{\epsilon_I} & \colim L \arrow{d}{\colim \epsilon} \\
				LI \arrow[swap]{r}{p_I} & \colim L
			\end{tikzcd}\end{center}
	\end{Lemma}
	\begin{proof}
		By the fact $\ov{p_I}=1$ and $\epsilon_I=\ov{\epsilon_I}$, we have
			$$ \ov{\colim\epsilon} \circ p_I = p_I \circ \ov{\colim\epsilon \circ p_I} = p_I \circ \ov{p_I \circ \epsilon_I} = 
				p_I \circ \ov{\ov{p_I} \circ\epsilon_I} = p_I \circ \ov{\epsilon_I}=p_I \circ \epsilon_I. $$
		Therefore, $\colim\epsilon = \ov{\colim\epsilon}$ by uniqueness.
	\end{proof}
	
	\begin{Lemma}\label{MCatLem2}
		Let $\cat{C}$ be a cocomplete $\M$-category, and let $H,K \colon \cat{I}\to\cat{C}$ be functors (with $\cat{I}$ small). 
		Suppose $\alpha\colon H \To K$ is a natural transformation such that for each $I\in\cat{I}$, $\alpha_I$ is in $\M_{\cat{C}}$ 
		and all naturality squares are pullbacks:
			\begin{center}\begin{tikzcd}
				HI \arrow{r}{Hf} \arrow[swap]{d}{\alpha_I} & HJ \arrow{d}{\alpha_J} \\
				KI \arrow[swap]{r}{Kf} & KJ
			\end{tikzcd}\end{center}
		Then $\colim\alpha$ is in $\M_{\cat{C}}$, and the following is a pullback for every $I\in\cat{I}$:
			\begin{center}\begin{tikzcd}
				HI \arrow{r}{p_I} \arrow[swap]{d}{\alpha_I} & \colim H \arrow{d}{\colim\alpha} \\
				KI \arrow[swap]{r}{q_I} & \colim K
			\end{tikzcd}\end{center}
		where $p_I,q_I$ are colimit coprojections.
	\end{Lemma}
	\begin{proof}
		Applying the inclusion $\imath\colon \cat{C}\to\Par(\cat{C})$ gives the following commutative diagram for each $I\in\cat{I}$:
			\begin{center}\begin{tikzcd}
				HI \arrow{r}{(1,p_I)} \arrow[swap]{d}{(1,\alpha_I)} & \colim H \arrow{d}{(1,\colim\alpha)} \\
				KI \arrow[swap]{r}{(1,q_I)} & \colim K
			\end{tikzcd}\end{center}
		Observe that there is a natural transformation $\beta\colon \imath K \To \imath H$ whose components 
		are given by $\beta_I= (\alpha_I,1)$; simply apply Lemma \ref{MCatLem0} to our assumption that $\alpha_I$ is a pullback of 
		$\alpha_J$ along $Kf$.
		
		Now the fact that the inclusion preserves the colimits $(\colim H,p_I)_{i\in\cat{I}}$ and
		$(\colim K,q_I)_{i\in\cat{I}}$ implies the existence of a unique map $\colim\beta=(n,g) \colon \colim K \to \colim H$ making 
		the following diagram commute for each $I\in\cat{I}$:
			\begin{center}\begin{tikzcd}
				KI \arrow{r}{(1,q_I)} \arrow[swap]{d}{(\alpha_I, 1)} & \colim K \arrow{d}{(n,g)} \\
				HI \arrow{r}{(1,p_I)} \arrow[swap]{d}{(1,\alpha_I)} & \colim H \arrow{d}{(1,\colim\alpha)} \\
				KI \arrow[swap]{r}{(1,q_I)} & \colim K
			\end{tikzcd}\end{center}
		Observe that the left composite $(1,\alpha_I)\circ(\alpha_I,1)=(\alpha_I,\alpha_I)$ is the component at $I$ of a 
		natural transformation $\epsilon\colon \imath K \To \imath K$ whose components are restriction idempotents. Therefore, 
		by Lemma~\ref{MCatLem1}, the composite on the right $(1,\colim\alpha)\circ(n,g)=(n,(\colim\alpha) g)$ must be a 
		restriction idempotent, and so $n=(\colim\alpha) g$.
		
		On the other hand, the composite $(\alpha_I,1)\circ(1,\alpha_I) = (1,1)$ is the component of the identity 
		natural transformation $\gamma\colon \imath H\To \imath H$ at $I$, and so $\colim\gamma\colon \colim H \to \colim H$ 
		must be $(1,1)$. However, as the following diagram also commutes, we must have $(n,g)\circ(1,\colim\alpha)=(1,1)$ 
		by uniqueness:
			\begin{center}\begin{tikzcd}
				HI \arrow{r}{(1,p_I)} \arrow[swap]{d}{(1,\alpha_I)} & \colim H \arrow{d}{(1,\colim\alpha)} \\
				KI \arrow{r}{(1,q_I)} \arrow[swap]{d}{(\alpha_I, 1)} & \colim K \arrow{d}{(n,g)} \\
				HI \arrow[swap]{r}{(1,p_I)} & \colim H
			\end{tikzcd}\end{center}
		So we have that $(1,\colim\alpha)\circ(n,g)=(n,n)$ is a splitting of the restriction idempotent $(n,n)$, which means that
		$(1,\colim\alpha)$ is a restriction monic. Therefore $\colim\alpha \in\M_{\cat{C}}$, proving the first part of the lemma.
		
		Regarding the second part of the lemma, observe that $(n,g)\circ(1,\colim\alpha)=(1,1)$ implies $g$ is an isomorphism
		(as $n=(\colim\alpha)$). Therefore, $(n,g)=(\colim\alpha,1)$ and so the following diagram commutes for all $I\in\cat{I}$:
			\begin{center}\begin{tikzcd}
				KI \arrow{r}{(1,q_I)} \arrow[swap]{d}{(\alpha_I,1)} & \colim K \arrow{d}{(\colim\alpha,1)} \\
				HI \arrow[swap]{r}{(1,p_I)} & \colim H
			\end{tikzcd}\end{center}
		The result then follows by applying Lemma \ref{MCatLem0}.
	\end{proof}
	
	\begin{Lemma}\label{MCatLem3}
		Let $\cat{C}$ be a cocomplete $\M$-category, $H,K \colon \cat{I}\to\cat{C}$ functors (with $\cat{I}$ small), and 
		$\alpha\colon H\To K$ a natural transformation such that each $\alpha_I \in\M_{\cat{C}}$ and all naturality squares
		are pullbacks (as in the previous lemma). Let $n\in\M_{\cat{C}}$, and suppose $x\colon \colim H \to X$ and 
		$y\colon\colim K \to Y$ make the right square commute and the outer square a pullback (for all $I\in\cat{I}$):
			\begin{center}\begin{tikzcd}
				HI \arrow{r}{p_I} \arrow[swap]{d}{\alpha_I} & \colim H \arrow{r}{x} \arrow{d}{\colim\alpha} & X \arrow{d}{n} \\
				KI \arrow[swap]{r}{q_I} & \colim K \arrow[swap]{r}{y} & Y
			\end{tikzcd}\end{center}
		Then the right square is also a pullback.
	\end{Lemma}
	\begin{proof}
		By Lemma \ref{MCatLem0}, to show that the right square is a pullback is the same as showing 
		$(1,x)\circ(\colim\alpha,1)=(\colim\alpha,x) = (n,1) \circ(1,y)$ in $\Par(\cat{C})$. In other words, that the top-right square 
		of the following diagram commutes:
			\begin{center}\begin{tikzcd}
				KI \arrow{r}{(1,q_I)} \arrow[swap]{d}{(\alpha_I,1)} & \colim K \arrow{r}{(1,y)} \arrow{d}{(\colim\alpha,1)} 
						& Y \arrow{d}{(n,1)} \\
				HI \arrow{r}{(1,p_I)} \arrow[swap]{d}{(1,\alpha_I)} & \colim H \arrow{r}{(1,x)} \arrow{d}{(1,\colim\alpha)} 
						& X \arrow{d}{(1,n)} \\
				KI \arrow[swap]{r}{(1,q_I)} & \colim K \arrow[swap]{r}{(1,y)} & Y
			\end{tikzcd}\end{center}
		Since $(\colim\alpha,x)$ and $(n,1)(1,y)$ are both maps out of $\colim K$, it is enough to show that
			$$ (\colim\alpha,x)(1,q_I)=(n,1)(1,y)(1,q_I) $$
		for all $I\in\cat{I}$. But the left-hand side is equal to $(\alpha_I,xp_I)$ by commutativity of the top-left square, and the 
		right-hand side is also $(\alpha_I,xp_I)$ by assumption. Hence the result follows.
	\end{proof}
	
	\begin{Cor}\label{MCatLem3Cor}
		If $(\cat{C},\M_{\cat{C}})$ is a cocomplete $\M$-category, then colimits in $\cat{C}$ are stable under pullback along 
		$\M_{\cat{C}}$-maps.
	\end{Cor}
	\begin{proof}
		Let $K\colon \cat{I}\to\cat{C}$ be a functor, $P$ any object in $\cat{C}$, and suppose $\mu\colon P \to \colim K$ is an 
		$\M_{\cat{C}}$-map. Since $\mu\in\M_{\cat{C}}$, for each $I\in\cat{I}$, we may take pullbacks 
		of $\mu$ along the colimiting coprojections of $\colim K$, $(k_I\colon K_I\to \colim K)_{I\in\cat{I}}$, and these
		we call $\alpha_I\colon HI \to KI$. This gives a functor $H\colon \cat{I}\to\cat{C}$, which on objects, takes $I$ to $HI$, 
		and on morphisms, takes $f\colon I\to J$ to the unique map making all squares in the following diagram pullbacks:
			\begin{center}\begin{tikzcd}
				HI \arrow{r}{Hf} \arrow[swap]{d}{\alpha_I} \arrow[bend left=30]{rr}{p_I} & HJ \arrow{r}{p_J} \arrow{d}{\alpha_J}
						& P \arrow{d}{\mu} \\
				KI \arrow{r}{Kf} \arrow[swap,bend right=20]{rr}{k_I} & KJ \arrow{r}{k_J} & \colim K
			\end{tikzcd}\end{center}
		By construction, $(P,p_I)_{I\in\cat{I}}$ is a cocone in $\cat{C}$ and $\alpha\colon H\to K$ is a natural transformation. 
		Now let $(h_I\colon HI\to\colim H)_{I\in\cat{I}}$ be the colimiting coprojections of $\colim H$. Then by the
		universal property of $\colim H$, there exists a unique $\gamma\colon\colim H\to P$ such that $p_I=\gamma h_I$ for all
		$I\in\cat{I}$, and by the universal property of $\colim K$, there is a $\colim \alpha\colon \colim H\to\colim K$ making
		the left square of the following diagram commute (for all $I\in\cat{I}$):
			\begin{center}\begin{tikzcd}
				HI \arrow{r}{h_I} \arrow[swap]{d}{\alpha_I} \arrow[bend left=30]{rr}{p_I} 
						& \colim H \arrow{r}{\gamma} \arrow{d}{\colim\alpha}
						& P \arrow{d}{\mu} \\
				KI \arrow{r} \arrow[swap,bend right=20]{rr}{k_I} & \colim K \arrow[-, double equal sign distance]{r} & \colim K
			\end{tikzcd}\end{center}
		It is easy to see that the right square commutes, and since the left square is a pullback for every $I\in\cat{I}$, the right
		square must be a pullback by Lemma \ref{MCatLem3}. Therefore, because the pullback of the identity $1_{\colim K}$ 
		is the identity, $P\cong \colim H$, and hence colimits are preserved by pullbacks along $\M_{\cat{C}}$-maps.
	\end{proof}
	
	We now show that for any small $\M$-category $\cat{C}$, the Yoneda embedding 
	$\yon\colon \cat{C}\to\PSh_{\M}(\cat{C})$ exhibits the $\M$-category of presheaves $\PSh_{\M}(\cat{C})$ as the free 
	cocompletion of $\cat{C}$.
	
	\begin{Thm}{(Free cocompletion of $\M$-categories)} \label{MCatCocomp}
		For any small $\M$-category $\cat{C}$ and cocomplete $\M$-category $\cat{D}$, the following is an equivalence of categories:
			\begin{equation}\label{FreeMCocomp}
				(-)\circ\yon\colon \M\cat{Cocomp}(\PSh_{\M}(\cat{C}),\cat{D}) \to \M\Cat(\cat{C},\cat{D}).
			\end{equation}
	\end{Thm}
	\begin{proof}
		We know that $(-)\circ\yon\colon \cat{Cocomp}(\PSh(\cat{C}),\cat{D}) \to \Cat(\cat{C},\cat{D})$ is an equivalence of 
		categories; that is, given a functor $F\colon\cat{C}\to\cat{D}$, there is a cocontinuous $G\colon \PSh(\cat{C})\to\cat{D}$ 
		such that $G\yon \cong F$. So \eqref{FreeMCocomp} will be essentially surjective on objects if this same $G$ is an 
		$\M$-functor.
		
		To see that $G$ takes monics in $\M_{\PSh(\cat{C})}$ to monics in $\M_{\cat{D}}$, let $\mu\colon P\To Q$ be an 
		$\M_{\PSh(\cat{C})}$-map. Since every presheaf is a colimit of representables, write $Q \cong \colim\yon D$, where 
		$D\colon\cat{I}\to\cat{C}$ is a functor (with $\cat{I}$ small). By definition of $\mu\in\M_{\PSh(\cat{C})}$, for every 
		$I\in\cat{I}$, there is a map $m_I\colon C_I\to D_I$ making the following a pullback:
			\begin{center}\begin{tikzcd}
				\yon C_I \arrow{r}{p_I} \arrow[swap]{d}{\yon m_I} & P \arrow{d}{\mu} \\
				\yon D_I \arrow[swap]{r}{q_I} & Q
			\end{tikzcd}\end{center}
		(where $q_I$ is a colimit coprojection). It follows there is a functor $C\colon \cat{I}\to\cat{C}$ which on objects takes 
		$I$ to $C_I$ and on morphisms, takes $f\colon I\to J$ to the unique map $Cf$ making the diagram below commute and the left 
		square a pullback:
			\begin{center}\begin{tikzcd}
				\yon C_I \arrow[dashed]{r}{\yon Cf} \arrow[swap]{d}{\yon m_I} \arrow[bend left=30]{rr}{p_I} 
						& \yon C_J \arrow{r}{p_J} \arrow[swap]{d}{\yon m_J} & P \arrow{d}{\mu} \\
				\yon D_I \arrow[swap]{r}{\yon Df} \arrow[swap,bend right=30]{rr}{q_I} & \yon D_J \arrow[swap]{r}{q_J} & Q
			\end{tikzcd}\end{center}
		The fact colimits in $\PSh(\cat{C})$ are stable under pullback implies $(p_I\colon \yon C_I \to P)_{I\in\cat{I}}$ is colimiting. 
		Now applying $G$ to the above diagram gives
			\begin{equation}\label{PbSquare}\begin{tikzcd}
				G\yon C_I \arrow{r}{G\yon Cf} \arrow[swap]{d}{G\yon m_I} \arrow[bend left=30]{rr}{Gp_I} & 
					G\yon C_J \arrow{r}{Gp_J} \arrow[swap]{d}{G\yon m_J} & GP \arrow{d}{G\mu} \\
				G \yon D_I \arrow[swap]{r}{G\yon Df} \arrow[swap,bend right=30]{rr}{Gq_I} & G \yon D_J \arrow[swap]{r}{Gq_J} & GQ
			\end{tikzcd}\end{equation}
		Since $G$ is cocontinuous, both $(Gp_I)_{I\in\cat{I}}$ and $(Gq_I)_{I\in\cat{I}}$ are colimiting. Also, as $G\yon \cong F$
		and $F$ is an $\M$-functor, the left square is a pullback for every pair $I,J \in \cat{I}$. Therefore, by Lemma \ref{MCatLem2}, 
		$G\mu$ must be in $\M_{\cat{D}}$. 
		
		Observe that the same lemma (Lemma \ref{MCatLem2}) says that for every $I\in\cat{I}$, the outer square in 
		\eqref{PbSquare} is a pullback for every $I\in\cat{I}$. In other words, $G$ preserves pullbacks of the form
				\begin{equation}\label{PreservePb}\begin{tikzcd}
					\yon C_I \arrow{r}{p_I} \arrow[swap]{d}{\yon m_I} & P \arrow{d}{\mu} \\
					\yon D_I \arrow[swap]{r}{q_I} & Q
				\end{tikzcd}\end{equation}
		
		Now to see that $G$ preserves $\M_{\PSh(\cat{C})}$-pullbacks, consider the diagram below, where
		the right square is an $\M_{\PSh(\cat{C})}$-pullback and the left square is a pullback for all $I\in\cat{I}$:
			\begin{center}\begin{tikzcd}
				\yon C_I \arrow{r}{p_I} \arrow[swap]{d}{\yon m_I} & P \cong \colim \yon C \arrow{r} \arrow{d}{\mu} & P' \arrow{d}{\mu'} \\
				\yon D_I \arrow[swap]{r}{q_I} & Q \cong \colim \yon D \arrow{r} & Q'
			\end{tikzcd}\end{center}
		The result then follows by applying $G$ to the diagram and using Lemma~\ref{MCatLem3}. This proves \eqref{FreeMCocomp}
		is essentially surjective on objects.
		
		Finally, to show that \eqref{FreeMCocomp} is fully faithful, we need to show for any cocontinuous pair of $\M$-functors 
		$F,F'\colon\PSh_{\M}(\cat{C})\to\cat{D}$ and $\M_{\cat{C}}$-cartesian $\alpha\colon F\yon \To F'\yon$, there exists a unique 
		$\M_{\PSh(\cat{C})}$-cartesian $\tilde{\alpha}\colon F \To F'$ such that $\tilde{\alpha}\yon=\alpha$. In other words,
		the following is an isomorphism of sets:
			$$ (-) \circ \yon \colon \M\Nat(F,F') \to \M\Nat(F\yon, F'\yon) $$
		where $\M\Nat(F,F')$ are the $\M$-cartesian natural tranformations between $F$ and $F'$. However, this condition may be 
		reformulated as follows: 
			\begin{equation}\label{reformulation}\begin{split}
				&\text{For all natural transformations $\tilde{\alpha}\colon F \To F'$, $\tilde{\alpha}$ is $\M_{\PSh(\cat{C})}$-cartesian if } \\
					&\text{$\tilde{\alpha}\yon\colon F\yon \To F'\yon$ is $\M_{\cat{C}}$-cartesian.}\end{split}
			\end{equation} 
		To see that these two statements are equivalent, observe that the second statement amounts to the following diagram
		being a pullback in $\Set$:
			\begin{center}\begin{tikzcd}
				\M\Nat(F,F') \arrow{r} \arrow[hook]{d} & \M\Nat(F\yon,F'\yon) \arrow[hook]{d} \\
				\Nat(F,F') \arrow[swap]{r}{(-) \circ \yon} & \Nat(F\yon,F'\yon)
			\end{tikzcd}\end{center}
		where $\Nat(F,F')$ is the set of natural transformations between $F$ and $F'$. However, as the bottom function 
		is an isomorphism (ordinary free cocompletion), the top must also be an isomorphism and hence the two statements are
		equivalent. Therefore, we show \eqref{FreeMCocomp} is fully faithful by proving \eqref{reformulation}.
		
		So let $\mu\colon P \To Q$ be an $\M_{\PSh(\cat{C})}$-map, and recall that the left square (diagram below) is a pullback for every 
		$I\in\cat{I}$ as $F$ preserves $\M_{\PSh(\cat{C})}$-pullbacks:
		\begin{equation}\label{MPCsquare}\begin{tikzcd}
			F\yon C_I \arrow{r}{Fp_I} \arrow[swap]{d}{F\yon m_I} & FP \arrow{r}{\tilde{\alpha}_P} \arrow{d}{F\mu} & F'P \arrow{d}{F'\mu} \\
			F\yon D_I \arrow[swap]{r}{Fq_I} & FQ \arrow[swap]{r}{\tilde{\alpha}_Q} & F'Q
		\end{tikzcd}\end{equation}
		To show that the right square is a pullback, we will show that the outer square is a pullback for every $I\in\cat{I}$ and apply
		Lemma \ref{MCatLem3}. Now by naturality of $\tilde{\alpha}$, this outer square is the outer square of the following diagram:
		\begin{center}\begin{tikzcd}
			F\yon C_I \arrow{r}{\tilde{\alpha}_{\yon {C_I}}} \arrow[swap]{d}{F \yon m_I} 
				& F'\yon C_I \arrow{r}{F'p_I} \arrow{d}{F' \yon m_I} 
				& F'P \arrow{d}{F' \mu} \\
			F\yon D_I \arrow[swap]{r}{\tilde{\alpha}_{\yon {D_I}}} & F'\yon D_I \arrow[swap]{r}{F'q_I} & F'Q
		\end{tikzcd}\end{center}
		But $\tilde{\alpha} \circ \yon$ being $\M_{\cat{C}}$-cartesian implies the left square is a pullback, and the right square is also 
		a pullback by the fact $F'$ preserves pullbacks of the form \eqref{PreservePb}. Therefore, by Lemma \ref{MCatLem3}, each
		square on the right of \eqref{MPCsquare} is a pullback, and so $\tilde{\alpha}$ is $\M_{\PSh(\cat{C})}$-cartesian. Hence, 
		$(-)\circ\yon\colon \M\cat{Cocomp}(\PSh_{\M}(\cat{C}),\cat{D}) \to \M\Cat(\cat{C},\cat{D})$ is an equivalence of categories.
	\end{proof}

\subsection{Cocompletion of restriction categories}
Earlier, we explored the notion of cocomplete $\M$-category. Now, by the fact $\M\Cat$ and $\cat{rCat}_s$ are $2$-equivalent, it 
makes sense to define a restriction category to be cocomplete in such a way that $\Par(\cat{C})$ will be cocomplete as a restriction 
category if and only if $\cat{C}$ is cocomplete as an $\M$-category. 

\begin{Defn}
	A restriction category $\cat{X}$ is cocomplete if it is split, its subcategory $\Total(\cat{X})$ is cocomplete, and the inclusion
	$\Total(\cat{X}) \hookrightarrow \cat{X}$ preserves colimits. A restriction functor $F\colon \cat{X}\to\cat{Y}$ is cocontinuous if 
	$\Total(F)\colon \Total(\cat{X})\to\Total(\cat{Y})$ is cocontinuous. We denote by $\cat{rCocomp}$ the $2$-category of
	cocomplete restriction categories, cocontinuous restriction functors and restriction transformations.
\end{Defn}

Observe that for any cocomplete restriction category $\cat{X}$, $\M\Total(\cat{X})$ is a cocomplete $\M$-category since 
$\Total(\cat{X})$ is cocomplete and $\Total(\cat{X}) \hookrightarrow \cat{X} \cong \Par(\M\Total(\cat{X}))$ preserves colimits.
We now give examples of cocomplete restriction categories.

\begin{Ex}
	For each class of examples from Example \ref{CocompMCat}, $\Par(\E,\M_{\E})$ is a cocomplete 
	restriction category. In particular, the restriction category of sets and partial functions $\Set_p$ is a cocomplete restriction 
	category since $\Set_p = \Par(\E,\M_{\E})$, where $\E = \Set$ and $\M_{\E}$ are the injective functions. 
	
	Also note that since the $\M$-category $(\cat{Ab},\M_{\cat{Ab}})$ of abelian groups and group monomorphisms is
	not cocomplete as an $\M$-category, $\Par(\cat{Ab},\M_{\cat{Ab}})$ is also not a cocomplete restriction category.
\end{Ex}

We know that for any small $\M$-category $\cat{C}$, $\PSh_{\M}(\cat{C})$ is a cocomplete $\M$-category, and furthermore,
$\Par(\PSh_{\M}(\cat{C}))$ is a cococomplete restriction category. In particular, the split restriction category
$\Par(\PSh_{\M}(\M\Total(\msf{K}_r(\cat{X}))))$ is a cocomplete restriction category for any small restriction category 
$\cat{X}$. We now show that the Cockett and Lack embedding below \cite[p.~252]{MR1871071}
	\begin{equation} \label{CL} 
		\cat{X} \xrightarrow{J} \msf{K}_r(\cat{X}) \xrightarrow{\Phi_{\msf{K}_r(\cat{X})}} \Par(\M\Total(\msf{K}_r(\cat{X})))
			\xrightarrow{\Par(\yon)} \Par(\PSh_{\M}(\M\Total(\msf{K}_r(\cat{X}))))
	\end{equation}
exhibits this split restriction category $\Par(\PSh_{\M}(\M\Total(\msf{K}_r(\cat{X}))))$ as the free restriction cocompletion 
of any small restriction category $\cat{X}$.

\begin{Thm}{(Free cocompletion of restriction categories)}
	For any small restriction category $\cat{X}$ and cocomplete restriction category $\E$, the following is an equivalence
	of categories:
		$$ (-)\circ \eqref{CL} \colon \cat{rCocomp}(\Par(\PSh_{\M}(\M\Total(\msf{K}_r(\cat{X})))),\E) \to 
			\cat{rCat}(\cat{X},\E) $$
	where \eqref{CL} is the Cockett and Lack embedding.
\end{Thm}
\begin{proof}
	First note that $\E \cong \Par(\cat{D})$ for some cocomplete $\M$-category $\cat{D}$ (as $\E$ is split), and that
		$$ \cat{rCocomp}(\Par(\PSh_{\M}(\cat{C})),\Par(\cat{D})) \simeq \M\cat{Cocomp}(\PSh_{\M}(\cat{C}),\cat{D}) $$
	since $\Par$ and $\M\Total$ are $2$-equivalences. Therefore, 
		$$ (-)\circ\Par(\yon)\colon\cat{rCocomp}(\Par(\PSh_{\M}(\cat{C})),\E) \to\cat{rCat}(\Par(\cat{C}),\E) $$
	is an equivalence since
		$$ (-)\circ\yon\colon \M\cat{Cocomp}(\PSh_{\M}(\cat{C}),\cat{D}) \to \M\Cat(\cat{C},\cat{D}) $$
	is an equivalence (free cocompletion of $\M$-categories). Therefore the following composite is an equivalence:
	\begin{center}\begin{tikzcd}
		\cat{rCocomp}(\Par(\PSh_{\M}(\M\Total(\msf{K}_r(\cat{X})))),\E) 
			\arrow{d}{(-) \circ \Par(\yon)} \\
		\cat{rCocomp}(\Par(\M\Total(\msf{K}_r(\cat{X}))),\E) \arrow{d}{(-) \circ \Phi_{\msf{K}_r(\cat{X})} \circ J} \\
		\cat{rCat}(\cat{X},\E)
	\end{tikzcd}\end{center}
	as $\Phi_{\msf{K}_r(\cat{X})}$ is an isomorphism and $J$ is the unit of the biadjunction $i \dashv \msf{K}_r$ at $\cat{X}$.
\end{proof}

\section{Restriction presheaves}\label{sec4}
We have just seen that for any small restriction category $\cat{X}$, the Cockett-Lack embedding \eqref{CL} exhibits the 
restriction category $\Par(\PSh_{\M}(\M\Total(\msf{K}_r(\cat{X}))))$ as a free cocompletion of $\cat{X}$. However, this formulation
of free cocompletion seems rather complex compared to the fact both $\PSh(\cat{C})$ and $\PSh_{\M}(\cat{C})$ were the 
free cocompletions of ordinary categories and $\M$-categories respectively.

In this section, we give an alternate simpler definition in terms of a restriction category of \emph{restriction presheaves} 
$\PSh_r(\cat{X})$. We shall see that $\PSh_r(\cat{X})$ is a full subcategory of $\PSh(\cat{X})$ and 
that the Yoneda embedding factors through a restriction functor $\yon_r\colon \cat{X} \to\PSh_r(\cat{X})$. Finally, we show that the
category of restriction presheaves $\PSh_r(\cat{X})$ is equivalent to $\Par(\PSh_{\M}(\M\Total(\msf{K}_r(\cat{X}))))$, so that
it gives another way of describing free cocompletion in the restriction setting.

\begin{Defn}{(Restriction presheaf)}
	Let $\cat{X}$ be a restriction category. A \emph{restriction presheaf} on $\cat{X}$ is a presheaf
	$P \colon \cat{X}^{\op} \to \Set$ together with assignations
		$$ PA \to \cat{X}(A,A), \quad x \mapsto \bar{x} $$
	where $\bar{x}$ is a restriction idempotent satisfying the following three axioms:
		\begin{enumerate}[label=(A\arabic*),leftmargin=1.5cm]
			\item $x \cdot \bar{x} = x$;
			\item $\ov{x \cdot \bar{f}} = \bar{x} \circ \bar{f}$, where $\bar{f} \colon A \to A$ is a restriction idempotent in $\cat{X}$;
			\item $\bar{x} \circ g = g \circ \ov{x \cdot g}$, where $g \colon B \to A$ in $\cat{X}$.
		\end{enumerate}
	The notation $x \cdot \bar{x}$ denotes the element $P(\bar{x})(x) \in PA$ \cite[p.~25]{MR1300636}. We call the assignations 
	$x \mapsto \bar{x}$ the \emph{restriction structure} on $P$.
\end{Defn}

Unlike the restriction structure on a restriction category, the restriction structure on any restriction presheaf is unique, due to 
the following lemma.

\begin{Lemma}
	Let $\cat{X}$ be a restriction category and $P \colon \cat{X}^{\op} \to \Set$ a presheaf. Suppose $P$ 
	has two restriction structures given by $x \mapsto \bar{x}$ and $x \mapsto \tilde{x}$. Then $\bar{x}=\tilde{x}$ for all
	$A\in\cat{X}$ and $x\in PA$.
\end{Lemma}
\begin{proof}
	By the fact $\bar{x}$ and $\tilde{x}$ are restriction idempotents and using (A1),(A2),
		\begin{align*} 
			\bar{x} &= \ov{x\cdot\tilde{x}} = \bar{x} \circ \tilde{x}=\tilde{x} \circ \bar{x}= \wt{x\cdot\bar{x}} = \tilde{x}. \qedhere
		\end{align*}
\end{proof}

We also have the following analogues of basic results of restriction categories.

\begin{Lemma}
	Suppose $P$ is a restriction presheaf on a restriction category $\cat{X}$, and let $A\in\cat{X}$, $x \in PA$ and $g \colon B \to A$. 
	Then
		\begin{enumerate}[label=(\arabic*),leftmargin=1.5cm]
			\item $\bar{g} \circ \ov{x \cdot g} = \ov{x \cdot g}$;
			\item $\ov{\bar{x} \circ g} = \ov{x \cdot g}$.
		\end{enumerate}
\end{Lemma}

\begin{proof}
	\begin{enumerate}
		\item By (R2), (A2) and (R1),
				$$ \bar{g} \circ \ov{x \cdot g} = \ov{x \cdot g} \circ \ov{g} = \ov{(x \cdot g) \cdot \bar{g}}
					= \ov{x \cdot (g \circ \bar{g})} = \ov{x \cdot g}. $$
		\item By (A3), (R3) and the previous result,
				\begin{align*} 
					\ov{\bar{x} \circ g} & = \ov{g \circ \ov{x \cdot g}} = \bar{g} \circ \ov{x \cdot g } = \ov{x \cdot g}. \qedhere 
				\end{align*}
	\end{enumerate}
\end{proof}

\begin{Defn}{(Category of restriction presheaves)}
	Let $\cat{X}$ be a restriction category. The category of restriction presheaves on $\cat{X}$, $\PSh_r(\cat{X})$, has the following data:
	\begin{description}[labelindent=1cm,leftmargin=1cm]
		\item [Objects] Restriction presheaves;
		\item [Morphisms] Arbitrary natural transformations $\alpha\colon P\Rightarrow Q$;
		\item [Restriction] The restriction of $\alpha\colon P \To Q$ is the natural transformation $\bar{\alpha}\colon P \To P$,
					given componentwise by
						$$ \bar{\alpha}_A(x) = x \cdot \ov{\alpha_A(x)} $$
					for every $A \in \cat{X}$ and $x\in PA$.
	\end{description}
\end{Defn}
Note that $\bar{\alpha}$ is natural since
				$$ \bar{\alpha}_B(x \cdot f) = x \cdot \Big( f \circ \ov{\alpha_B(x \cdot f)} \Big) 
						= x \cdot \Big( f \circ \ov{\alpha_A(x) \cdot f} \Big) = x \cdot \Big( \ov{\alpha_A(x)} \circ f \Big) 
						= \bar{\alpha}_A(x) \cdot f $$
for all $f \colon B \to A$. The restriction axioms are easy to check.

Observe that $\PSh_r(\cat{X})$ is a full subcategory of $\PSh(\cat{X})$, as the restriction structure on any presheaf,
if it exists, must be unique. Also, if $\cat{X}$ is a restriction category, then each representable $\cat{X}(-,A)$ has a restriction 
structure given by sending $f \in\cat{X}(B,A)$ to $\bar{f} \in\cat{X}$. In particular, this implies that the Yoneda embedding 
$\yon \colon \cat{X} \to \PSh(\cat{X})$ factors as a unique functor $\yon_r \colon \cat{X} \to \PSh_r(\cat{X})$.
	\begin{center}\begin{tikzcd}
		\cat{X} \arrow[swap]{dr}{\yon} \arrow{r}{\yon_r} & \PSh_r(\cat{X}) \arrow[hook]{d} \\
		& \PSh(\cat{X})
	\end{tikzcd}\end{center}

\begin{Lemma}
	For any restriction category $\cat{X}$, the functor $\yon_r \colon \cat{X} \to \PSh_r(\cat{X})$ is a restriction functor.
\end{Lemma}
\begin{proof}
	Let $f\colon A \to B$ be a map in $\cat{X}$. Then for all $X \in \cat{X}$ and $x\in \cat{X}(X,A)$, we have 
		$$ \ov{\yon_r f}_X(x) = x \cdot \ov{(\yon_r f)_X(x)} = x \cdot \ov{f \circ x} = x \circ \ov{f \circ x} 
			= \ov{f} \circ x = (\yon_r \ov{f})_X (x) $$
	and so $\yon_r$ is a restriction functor.
\end{proof}

We can characterise the total maps in $\PSh_r(\cat{X})$ as those which are restriction preserving, due to the following proposition.

\begin{Prop}\label{TotalMaps}
	A map $\alpha \colon P \To Q$ is total in $\PSh_r(\cat{X})$ if and only if 
	$\ov{\alpha_A(x)} = \bar{x}$ for all $A \in \cat{X}$ and $x \in PA$.
\end{Prop}
\begin{proof}	
	Suppose $\alpha \colon P \To Q$ is total in $\PSh_r(\cat{X})$. Then $\bar{\alpha}_A(x) = 1_{PA}(x) = x$, or 
	$x \cdot \ov{\alpha_A(x)} = x$. But this implies $\bar{x} \le \ov{\alpha_A(x)}$ since
		$$ \bar{x} = \ov{x \cdot \ov{\alpha_A(x)}} = \bar{x} \circ \ov{\alpha_A(x)} = \ov{\alpha_A(x)} \circ \bar{x}$$
	On the other hand, $\ov{\alpha_A(x)} \le \bar{x}$ as
		$$ \ov{\alpha_A(x)} = \ov{\alpha_A(x \cdot \ov{x})} = \ov{\alpha_A(x) \cdot \bar{x}} = \ov{\alpha_A(x)} \circ \bar{x}
				= \bar{x} \circ \ov{\alpha_A(x)} $$
	Therefore, $\alpha$ in $\PSh_r(\cat{X})$ is total if and only if $\alpha$ preserves restrictions.	
\end{proof}

The restriction presheaf category has one more important property.

\begin{Prop}\label{SplitRCat}
	Let $\cat{X}$ be a restriction category. Then $\PSh_r(\cat{X})$ is a split restriction category.
\end{Prop}
\begin{proof}
	Let $\bar{\alpha} \colon P \To P$ be a restriction idempotent in $\PSh_r(\cat{X})$. Since all idempotents in $\PSh(\cat{X})$
	split, we may write $\bar{\alpha}=\mu\rho$ for some maps $\mu\colon Q \To P$ and $\rho\colon P \To Q$ such that 
	$\rho\mu=1$. Componentwise, we may take $\mu_A$ to be the inclusion $QA \hookrightarrow PA$ and 
	$QA = \{ x\in PA \mid \bar{\alpha}_A(x)=x \}$. 
	Therefore, to show $\PSh_r(\cat{X})$ is split, it is enough to show that $Q$ is a restriction presheaf. However, $P$ is a 
	restriction presheaf and $Q$ is a subfunctor of $P$. Therefore, imposing the restriction structure of $P$ onto $Q$ will make 
	$Q$ a restriction presheaf. Hence $\PSh_r(\cat{X})$ is a split restriction category.
\end{proof}

Before moving onto the main theorems in this section, let us recall the split restriction category $\msf{K}_r(\cat{X})$, whose 
objects are pairs $(A,e)$ (with $e$ a restriction idempotent on $A \in\cat{X}$). Also recall the unit of the biadjunction
$i \dashv \msf{K}_r$ at $\cat{X}$, $J \colon \cat{X} \to \msf{K}_r(\cat{X})$, which sends objects $A$ to $(A,1_A)$ and 
morphisms $f \colon A \to B$ to $f \colon (A,1_A) \to (B,1_B)$.

\begin{Prop}
	$\PSh_r(\cat{X})$ and $\PSh_r(\msf{K}_r(\cat{X}))$ are equivalent as restriction categories.
\end{Prop}
\begin{proof}
	It is well-known that the functor $(-) \circ J^{\op} \colon \PSh(\msf{K}_r(\cat{X})) \to \PSh(\cat{X})$ is an equivalence. Therefore, 
	the result will follow if we can show this functor restricts back to an equivalence between $\PSh_r(\msf{K}_r(\cat{X}))$ and 
	$\PSh_r(\cat{X})$. In other words, showing that the restriction of $(-) \circ J^{\op}$ to $\PSh_r(\msf{K}_r(\cat{X}))$ 
	sends objects to restriction presheaves, is essentially surjective on objects and is a restriction functor.
		\begin{center}\begin{tikzcd}
			\PSh_r(\msf{K}_r(\cat{X})) \arrow{r} \arrow[hook]{d} & \PSh_r(\cat{X}) \arrow[hook]{d} \\
			\PSh(\msf{K}_r(\cat{X})) \arrow{r}{(-) \circ J^{\op}} & \PSh(\cat{X})
		\end{tikzcd}\end{center}
	So let $P$ be a restriction presheaf on $\msf{K}_r(\cat{X})$. Then $PJ^{\op}$ will be a restriction presheaf on $\cat{X}$ if we 
	define the restriction on $x \in (PJ^{\op})(A)=P(A,1_A)$ to be the same as in $P(A,1_A)$ for all $A\in\cat{X}$. Also, if 
	$\bar{\alpha}\colon P\To P$ is a restriction idempotent, then
		$$ (\ov{\alpha} \circ J^{\rm{op}})_A(x) =  \ov{\alpha}_{(A,1_A)} (x) = x \cdot \ov{\alpha_{(A,1_A)}(x)}
				= x \cdot \ov{(\alpha \circ J^{\rm{op}})_A(x)} = \Big(\ov{\alpha \circ J^{\rm{op}} }\Big)_A(x) $$
	implies $(-) \circ J^{\op}$ is a restriction functor. Therefore, all that remains is to show essential surjectivity.
	
	Let $Q$ be a restriction presheaf on $\cat{X}$, and define a presheaf $Q'$ on $\msf{K}_r(\cat{X})$ as follows:
		\begin{description}[labelindent=1cm]
			\item [Objects] $(A,e) \mapsto \{ x\in QA \mid x\cdot e = x\}$;
			\item [Morphisms] $f \colon (A,e) \to (A',e') \mapsto Qf$.
		\end{description}
	A quick check will show that $Q' \circ J^{\op} = Q$. Then to make $Q'$ a restriction presheaf, observe that because
	$x \in Q'(A,e) \subseteq QA$ and $Q$ is a restriction presheaf, there exists a restriction idempotent $\bar{x}$ associated to $x$.
	Therefore, define the restriction structure on $Q'$ to be $x \mapsto \bar{x}$. This then will satisfy the restriction presheaf axioms,
	making $Q'$ a restriction presheaf. Hence, $(-) \circ J^{\op} \colon \PSh(\msf{K}_r(\cat{X})) \to \PSh(\cat{X})$ is
	essentially surjective on objects, and so $\PSh_r(\cat{X})$ and $\PSh_r(\msf{K}_r(\cat{X}))$ are equivalent.
\end{proof}

Before proving that $\Par(\PSh_{\M}(\M\Total(\msf{K}_r(\cat{X}))))$ and $\PSh_r(\cat{X})$ are, in fact, equivalent as 
restriction categories, we give the following lemma.

\begin{Lemma}\label{PbLemma}
	Let $\cat{C}$ be a category and let $m$ be a monic in $\cat{C}$. Suppose the following is a pullback:
		\begin{center}\begin{tikzcd}
			D \arrow{r}{g} \arrow[swap]{d}{n} & A \arrow{d}{m} \\
			B \arrow[swap]{r}{f} & C 
		\end{tikzcd}\end{center}
	Then $n$ is an isomorphism if and only if $f=mh$ for some $h \colon B\to A\in\cat{C}$.
\end{Lemma}
\begin{proof}
	($\Rightarrow$) Take $h=gn^{-1}$. ($\Leftarrow$) Consider maps $1_B\colon B \to B$ and 
	$h\colon B\to A$ and use the fact the square is a pullback.
\end{proof}

We now give the following equivalence of $\M$-categories.

\begin{Thm}\label{ThmEquiv}
	Suppose $\cat{C}$ is an $\M$-category. Then $\M\Total(\PSh_r(\Par(\cat{C})))$ and $\PSh_{\M}(\cat{C})$ are equivalent.
\end{Thm}
\begin{proof}
	We adopt the following approach. First, find functors $F \colon \PSh(\cat{C}) \to \Total(\PSh_r(\Par(\cat{C})))$ and
	$G \colon \Total(\PSh_r(\Par(\cat{C}))) \to \PSh(\cat{C})$, and natural isomorphisms $\eta\colon1 \To GF$ and 
	$\epsilon\colon FG \To 1$. We then show that $F$ and $G$ are in fact $\M$-functors. (Note that $\eta$ and $\epsilon$ must 
	necessarily be $\M$-cartesian).
	
	So let $P$ be a presheaf on $\cat{C}$, and define $F$ on objects as follows. If $X \in \Par(\cat{C})$, then $(FP)(X)$ is the set of 
	equivalence classes
		$$ (FP)(X) = \{ (m,f) \mid m\colon Y \to X \in \M_{\cat{C}}, f \in PY \} $$
	where $(m,f) \sim (n,g)$ if and only if there exists an isomorphism $\phi$ such that $n=m\phi$ and $g=f\cdot\phi$. 
	To define $FP$ on morphisms, given $(n,g) \colon Z \to X$ in $\Par(\cat{C})$ and an element $(m,f) \in (FP)(X)$, define
		$$ \Big( (FP)(n,g) \Big) (m,f) = (nm', f\cdot g') $$
	where $(m',g')$ is the pullback of $(m,g)$, as in:
		\begin{center}\begin{tikzcd}
			\bullet \arrow{r}{g'} \arrow[swap]{d}{m'} & \bullet \arrow{d}{m} \\
			\bullet \arrow[swap]{r}{g} & \bullet
		\end{tikzcd}\end{center}
	We shall sometimes denote the above informally as $(m,f) \cdot (n,g)$. Then defining the restriction on each 
	$(m,f) \in (FP)(X)$ to be $(m,m)$ makes $FP \colon \Par(\cat{C})^{\op} \to \Set$ a restriction presheaf. This defines $F$
	on objects. 
	
	Now suppose $\alpha \colon P \To Q$ is a morphism in $\PSh(\cat{C})$. Define $F\alpha \colon FP \to FQ$ componentwise 
	as follows:
		$$ (F\alpha)_X(m,f) = (m,\alpha_{\dom \,m} (f)). $$
	Then $F\alpha$ is natural (by naturality of $\alpha$) and also total, making $F$ a functor from $\PSh(\cat{C})$ to 
	$\Total(\PSh_r(\Par(\cat{C})))$. We now give the data for the functor $G$ from $\Total(\PSh_r(\Par(\cat{C})))$ to $\PSh(\cat{C})$.
	
	Let $P$ be a restriction presheaf on $\Par(\cat{C})$, and define $GP\colon\cat{C}^{\op}\to\Set$ as follows. If $X \in\cat{C}$, then
		$$ (GP)(X) = \{ x \mid x \in PX, \bar{x} = (1,1) \}. $$
	And if $f\colon Z \to X$ is an arrow in $\cat{C}$, define
		$$ (GP)(f) = P(1,f). $$
	Note that $(GP)(f)$ is well-defined since for every $x\in (GP)(X)$,
		$$ \ov{P(1,f)(x)} = \ov{x\cdot(1,f)} = \ov{\bar{x} \circ (1,f)} = 1, $$
	and so $(GP)(f)$ is a function from $(GP)(X)$ to $(GP)(Z)$.
	
	Finally, if $\alpha \colon P \To Q$ is a total map in $\PSh_r(\Par(\cat{C}))$, define $G\alpha\colon GP \To GQ$ componentwise by
		$$ (G\alpha)_X(x) = \alpha_X(x) $$
	for every $X \in \cat{C}$ and $x\in (GP)(X)$. Again, to see that $G\alpha$ is well-defined, note that $\alpha$ total implies
	$\ov{\alpha_X(x)} = \ov{x} = 1$ (Proposition \ref{TotalMaps}) and so $\alpha_X(x) \in (GQ)(X)$.
	This makes $G$ a functor from $\Total(\PSh_r(\Par(\cat{C})))$ to $\PSh(\cat{C})$. The next step is defining 
	isomorphisms $\eta \colon 1 \To GF$ and $\epsilon \colon FG \To 1$.
	
	To define $\eta$, we need to give components for every presheaf $P$ on $\cat{C}$, and this involves giving isomorphisms 
	$(\eta_P)_X \colon PX \to (GFP)(X)$. But $(GFP)(X) = \{ (1,f) \mid f\in PX \}$. Therefore, defining $(\eta_P)_X(f) = (1,f)$ makes
	$\eta$ an isomorphism, and naturality is easy to check.
	
	Similarly, to define $\epsilon$, we need to define isomorphisms $(\epsilon_P)_X \colon (FGP)(X) \to PX$ for every restriction 
	presheaf $P$ on $\Par(\cat{C})$ and object $X \in \Par(\cat{C})$. Since 
		$$ (FGP)(X) = \{ (m,f) \mid m \colon Y \to X \in \M_{\cat{C}}, f\in PY, \bar{f}=(1,1) \}, $$
	define $(\epsilon_P)_X(m,f) = f \cdot (m,1)$. Its inverse $(\epsilon_P)^{-1}_X \colon PX \to (FGP)(X)$ is then given by
		$$ (\epsilon_P)^{-1}_X(x) = (n, x\cdot (1,n)) $$
	where $\bar{x} = (n,n)$ (as $P$ is a restriction presheaf on $\Par(\cat{C})$). Checking the naturality of $\epsilon$ is 
	again straightforward. All that remains is to show that both $F \colon \PSh_{\M}(\cat{C}) \to \M\Total(\PSh_r(\Par(\cat{C})))$ 
	and $G \colon \M\Total(\PSh_r(\Par(\cat{C}))) \to \PSh_{\M}(\cat{C})$ are $\M$-functors. However, as $F$ and $G$ are 
	equivalences in $\Cat$, they necessarily preserve limits, and so all this will involve is showing that they preserve $\M$-maps. 
	That is, $F\mu$ is a restriction monic in $\PSh_r(\Par(\cat{C}))$ for all $\mu \in \M_{\PSh(\cat{C})}$, and that 
	$G\mu$ is in $\M_{\PSh(\cat{C})}$ for all restriction monics $\mu \in \PSh_r(\Par(\cat{C}))$.
	
	So let $\mu \colon P \To Q$ be in $\M_{\PSh(\cat{C})}$. To show $F\mu$ is a restriction monic, we need to show $F\mu$ is the
	equaliser of $1$ and some restriction idempotent $\alpha \colon FQ \To FQ$. To define this $\alpha$, let
	$X \in \Par(\cat{C})$ and $(n,g) \in (FQ)(X)$ (where $n \colon Z \to X$). Now as $g\in QZ$, there exists a corresponding natural
	transformation $\hat{g}\colon\yon Z \To Q$ (Yoneda). However, as $\mu$ is in $\M_{\PSh(\cat{C})}$, there exists an
	$m_g \colon B \to Z$ in $\M_{\cat{C}}$ making the following a pullback:
		\begin{center}\begin{tikzcd}
			\yon B \arrow{r} \arrow[swap]{d}{\yon m_g} & P \arrow{d}{\mu} \\
			\yon Z \arrow[swap]{r}{\hat{g}} & Q
		\end{tikzcd}\end{center}
	So define $\alpha$ by its components as follows,
		$$ \alpha_X(n,g) = (nm_g, g \cdot m_g). $$
	It is then not difficult to show this $\alpha$ is well-defined, is a natural transformation and is a restriction idempotent.
	
	Now to show that $F\mu$ equalises $1$ and $\alpha$, we need to show $(F\mu)_X\colon (FP)(X) \to (FQ)(X)$ is an equaliser of 
	$1$ and $\alpha_{(FQ)(X)}$ in $\Set$ for all $X\in\Par(\cat{C})$. In other words, that $(F\mu)_X$ is injective, and that:
	\begin{equation}\label{FmuXinj}\begin{split}
		& \text{$(n,g)\in (FQ)(X)$ satisfies $(n,g)=(F\mu)_X(m,f)=(m,\mu_{\dom \,m}(f))$ for some}\\
		& \text{$(m,f)\in (FP)(X)$ if and only if $\alpha_X(n,g) = (n,g)$}.
	\end{split}\end{equation}
	
	To show $(F\mu)_X$ is injective, suppose $(F\mu)_X(m,f) = (F\mu)_X(m',f')$, or equivalently,
	$(m,\mu_{\dom\,m}(f)) = (m',\mu_{\dom\,m'}(f'))$. That is, there exists an isomomorphism $\varphi$ such that
	$m'=m\varphi$ and $\mu_{\dom\,m'}(f')=\mu_{\dom\,m}(f)\cdot \varphi$. But the naturality of $\mu$ implies
	$\mu_{\dom\,m'}(f\cdot\varphi)=\mu_{\dom\,m}(f)\cdot \varphi = \mu_{\dom\,m'}(f')$. Therefore, as $\mu$ is monic,
	we must have $f\cdot\varphi=f'$. Hence $(m,f)=(m',f')$, and so $(F\mu)_X$ is injective.
	
	To prove \eqref{FmuXinj}, let $(n,g)\in(FQ)(X)$ and suppose $\mu_X(n,g)=(n,g)$. That is, $(nm_g,g\cdot m_g)=(n,g)$, or
	that $m_g$ is an isomorphism. Now $m_g$ is an isomorphism if and only if $\yon m_g$ is an isomorphism, and by 
	Lemma \ref{PbLemma}, $\yon m_g$ is an isomorphism if and only if $\hat{g}=\mu \hat{h}$ for some $\hat{h}\colon\yon Z\to P$:
		\begin{center}\begin{tikzcd}
			\yon B \arrow{r} \arrow[swap]{d}{\yon m_g} & P \arrow{d}{\mu} \\
			\yon Z \arrow[swap]{r}{\hat{g}} \arrow[dashed]{ur}{\hat{h}} & Q
		\end{tikzcd}\end{center}
	But by Yoneda, the statement $\hat{g}=\mu\hat{h}$ is equivalent to the statement that $g=\mu_Z(h)$ for some $h\in PZ$,
	which is the same as saying $(n,g) = (n,\mu_Z(h)) = (F\mu)_X(n,h)$, with $(n,h) \in (FP)(X)$. Therefore, $(F\mu)_X$ is
	an equaliser of $1$ and $\alpha_{(FQ)(X)}$ in $\Set$ for all $X\in\Par(\cat{C})$, and hence, $F\mu$ equalises $1$ and $\alpha$.
	
	Now to see that $G$ is also an $\M$-functor, let $\mu \colon P \To Q$ be a restriction monic in $\PSh_r(\Par(\cat{C}))$. To show
	$G\mu$ is in $\M_{\PSh(\cat{C})}$, we need to show for any given $\hat{\theta} \colon \yon C \To Q$, there exists a monic 
	$m \colon D \to C$ in $\M_{\cat{C}}$ and a map $\hat{\delta} \colon \yon D \To P$ making the following a pullback:
		\begin{center}\begin{tikzcd}
			\yon D \arrow{r}{\hat{\delta}} \arrow[swap]{d}{\yon m} & GP \arrow{d}{G\mu} \\
			\yon C \arrow[swap]{r}{\hat{\theta}} & GQ
		\end{tikzcd}\end{center}
	Here we make two observations. First, commutativity says $m$ and $\delta$ must satisfy 
	$G\mu \circ \hat{\delta} = \hat{\theta} \circ \yon m$. On the other hand, Yoneda tells us that 
	$\hat{\theta} \circ \yon m = \wh{\theta \cdot m}$ and $G\mu \circ \hat{\delta} = \wh{ (G\mu)_D(\delta)}$, where 
	$\theta\in QC$ and $\delta\in PD$ are the unique transposes of $\hat{\theta}$ and $\hat{\delta}$ respectively. Therefore, 
	$m$ and $\delta$ must satisfy the following condition:
		\begin{equation} \label{eq3a} (G\mu)_D(\delta) = \theta \cdot_{GQ} m. \end{equation}
	That is, $\mu_D(\delta) = \theta \cdot_Q(1,m)$. Secondly, $m$ and $\delta$ must make the following a pullback in 
	$\Set$ (for all objects $X \in \cat{C}$):
		\begin{center}\begin{tikzcd}[column sep=huge]
			\cat{C}(X,D) \arrow{r}{\hat{\delta}_X = \delta \cdot_{GP} (-)} \arrow[swap]{d}{m \circ (-)} 
					& (GP)(X) \arrow{d}{(G\mu)_X} \\
			\cat{C}(X,C) \arrow[swap]{r}{\hat{\theta}_X = \theta \cdot_{GQ} (-)} & (GQ)(X)
		\end{tikzcd}\end{center}
	In other words, for any $f \in \cat{C}(X,C)$ and $x \in (GP)(X)$ such that $\theta \cdot_{GQ}f = (G\mu)_X(x)$
	(i.e., such that $\theta \cdot_Q (1,f) = \mu_X(x)$), there exists a unique $g\in\cat{C}(X,D)$ such that
		\begin{equation} \label{eq3b} \delta \cdot_{GP} g = x, \text{ and } mg=f. \end{equation}
	Alternatively, $\delta \cdot_P (1,g) = x$ and $mg=f$. To find $m$, note that because $\mu$ is a 
	restriction monic, there exists a $\rho$ such that $\mu \rho= \bar{\rho}$ and $\rho\mu=1$. Since $\theta \in QC$, applying 
	$\rho_C$ to $\theta$ and then taking its restriction gives $\ov{\rho_C(\theta)} = (m,m)$ for some $m \in \M_{\cat{C}}$. This gives us 
	$m$.
	
	To define $\delta$, observe that $P(1,m)$ is a function from $PC$ to $PD$. So define
		$$ \delta = \rho_C(\theta) \cdot_P (1,m). $$
	Then $\delta \in (GP)(D)$ since
		$$ \bar{\delta} = \ov{\ov{\rho_C(\theta)} \circ (1,m)} = \ov{(m,m) \circ (1,m)} = \ov{(1,m)} = (1,1). $$
	So all that remains is to show $m$ and $\delta$ satisfy \eqref{eq3a} and \eqref{eq3b}. To show $m$ and $\delta$ satisfy
	\eqref{eq3a}, one simply substitutes the given values into the equation, using the fact $\mu\rho=\bar{\rho}$. To see that 
	\eqref{eq3b} is also satisfied, suppose there exists an $f\in\cat{C}(X,C)$ and $x\in(GP)(X)$ such that 
	$\theta\cdot_P(1,f)=\mu_X(x)$. Then applying $\rho_X$ to both sides gives
		$$ \rho_C(\theta) \cdot_P (1,f) = x $$
	since $\rho\mu=1$. We need to show there exists a $g$ such that $mg=f$ and $\delta\cdot_P(1,g)=x$. But $mg=f$ implies
		$$ x = \rho_C(\theta) \cdot_P (1,f) = \rho_C(\theta) \cdot_P (1,mg) = \rho_C(\theta) \cdot_P (1,m) \cdot_P (1,g) 
				= \delta \cdot_P (1,g) $$
	Therefore, we just need to find $g$.
	
	Consider the composite $(m,m) \circ (1,f) = (m',mf')$, where $(m',f')$ is the pullback of $(m,f)$:
		\begin{center}\begin{tikzcd}
			X \times_C D \arrow{r}{f'} \arrow[swap]{d}{m'} & D \arrow{d}{m} \\
			X \arrow[swap]{r}{f} & C
		\end{tikzcd}\end{center}
	Note that if $m'$ is an isomorphism, then $g=f'(m')^{-1}$ will satisfy the condition $mg=f$. Now by restriction presheaf axioms and 
	naturality of $\bar{\rho}$, we have $\theta \cdot_Q (m',mf') = \theta\cdot_Q(1,f)$. But $\theta\in(GQ)(C)$ implies
		$$ \ov{\theta\cdot_Q(m',mf')} = \ov{\bar{\theta} \circ (m',mf')} = \ov{(m',mf')} = (m',m') $$
	and
		$$ \ov{\theta\cdot_Q(1,f)} = \ov{\bar{\theta}\circ(1,f)} = \ov{(1,f)} = (1,1). $$
	Therefore, $m'$ must be an isomorphism, which means $m$ and $\delta$ satisfy \eqref{eq3b}. Hence, $G$ is also an $\M$-functor and
	$\PSh_{\M}(\cat{C})$ and $\M\Total(\PSh_r(\Par(\cat{C})))$ are equivalent.
\end{proof}

We now use the above theorem to prove the following result.

\begin{Prop}\label{equivRestr}
	Let $\cat{C}$ be an $\M$-category. Then there exists an equivalence of restriction categories 
	$L \colon \Par(\PSh_{\M}(\cat{C})) \to \PSh_r(\Par(\cat{C}))$ satisfying the relation $\yon_r = L \circ \Par(\yon)$.
\end{Prop}
\begin{proof}
	Since $\Par$ and $\M\Total$ are 2-equivalences, the following is an isomorphism of categories: \small
		$$ \M\Cat\Big( \PSh_{\M}(\cat{C}), \M\Total(\PSh_r(\Par(\cat{C}))) \Big)
			\cong \cat{rCat}\Big( \Par(\PSh_{\M}(\cat{C})), \PSh_r(\Par(\cat{C})) \Big). $$ \normalsize
	We know from Theorem \ref{ThmEquiv} that $F \colon \PSh_{\M}(\cat{C}) \to \M\Total(\PSh_r(\Par(\cat{C})))$ is an 
	equivalence. So define $L=\tilde{F}$, the transpose of $F$. Explicitly, 
	$\tilde{F} = \Phi^{-1}_{\PSh_r(\Par(\cat{C}))} \circ \Par(F)$, where $\Phi_{\PSh_r(\Par(\cat{C}))}$ is the unit of the 
	$\Par$ and $\M\Total$ $2$-equivalence.
	
	Now define $\tilde{\yon}_r \colon \cat{C} \to \M\Total(\PSh_r(\Par(\cat{C})))$ as the transpose of 
	$\yon_r \colon \Par(\cat{C}) \to \PSh_r(\Par(\cat{C}))$. Explicitly, $\tilde{\yon}_r$ is the unique map making the following diagram 
	commute:
		\begin{center}\begin{tikzcd}
			\cat{C} \arrow{r}{\tilde{\yon}_r} \arrow[hook]{d} & \M\Total(\PSh_r(\Par(\cat{C}))) \arrow[hook]{d} \\
			\Par(\cat{C}) \arrow[swap]{r}{\yon_r} & \PSh_r(\Par(\cat{C}))
		\end{tikzcd}\end{center}
	Since $\tilde{\yon}_r = F \yon$ will imply $\yon_r = L \circ \Par(\yon)$, we prove the former. So let $A\in\Par(\cat{C})$. Then 
	$\tilde{\yon}_r(A) = \Par(\cat{C})(-,A)$ by definition. On the other hand, $(F\yon)(A)$ defined on objects $B\in\Par(\cat{C})$ 
	is the following set:
		$$ (F\yon A)(B) = \{ (m,f) \mid m\colon Y \to B \in\M_{\cat{C}}, f \in\cat{C}(Y,A) \}. $$
	In other words, elements of $(F\yon A)(B)$ are spans
		$$ B \xleftarrow{m} Y \xrightarrow{f} A $$
	Clearly $(F\yon A)(B) = \Par(\cat{C})(B,A) = (\tilde{\yon}_r A)(B)$. Likewise, if $(n,g) \colon C \to B$ is a morphism in 
	$\Par(\cat{C})$, then $(F\yon A)(n,g) = (-) \circ (n,g) = (\tilde{\yon}_r A)(n,g)$, and so $\tilde{\yon}_r(A) = (F\yon)(A)$.
	
	Now let $h \colon B \to C$ be a morphism in $\cat{C}$. Then $(F\yon)(h) \colon \Par(\cat{C})(-,B) \To \Par(\cat{C})(-,C)$ 
	has components given by
		$$ (F\yon h)_D(n,g) = (n,(\yon h)_{\dom \, n}(g)) = (n,hg) = (1,h) \circ (n,g) $$
	for all $D \in\Par(\cat{C})$ and $(n,g)\in\Par(\cat{C})(D,C)$. But $\tilde{\yon}_r(h) = \yon_r(1,h)$ also has components 
	given by $\Big( \yon_r(1,h) \Big)_D = (1,h) \circ (-)$ at $D\in\Par(\cat{C})$. Therefore, $(F\yon)(h)=\tilde{\yon}_r(h)$ and so 
	$F\yon = \tilde{\yon}_r$. Hence, $\yon_r = L \circ \Par(\yon)$.
\end{proof}

We now prove the main result of this section.
\begin{Thm}\label{CLembed}
	Let $\cat{X}$ be a restriction category. Then
		$$ \PSh_r(\cat{X}) \simeq \Par(\PSh_{\M}(\M\Total(\msf{K}_r(\cat{X})))) $$
	and the following diagram commutes up to isomorphism:
		\begin{center}\begin{tikzcd}[row sep=small,column sep=small]
			& \cat{X} \arrow[swap]{ddl}{\yon_r} \arrow{ddr}{\eqref{CL}} & \\
			& \cong & \\
			\PSh_r(\cat{X}) & & \Par(\PSh_{\M}(\M\Total(\msf{K}_r(\cat{X})))) \arrow{ll}{\simeq} 
		\end{tikzcd}\end{center}
	where \eqref{CL} is the Cockett and Lack embedding.
\end{Thm}
\begin{proof}
	Consider the following diagram, where $\cat{C} = \M\Total(\msf{K}_r(\cat{X}))$ and the top composite is \eqref{CL}:
		\begin{center}\begin{tikzcd}[column sep=huge]
			\cat{X} \arrow[swap]{d}{\yon_r} \arrow{r}{\Phi_{\msf{K}_r(\cat{X})} \circ J} & \Par(\cat{C}) \arrow{r}{\Par(\yon)} \arrow{d}{\yon_r}
					& \Par(\PSh_{\M}(\cat{C})) \arrow{d}{L} \\
			\PSh_r(\cat{X}) & \PSh_r(\Par(\cat{C})) \arrow[-, double equal sign distance]{r} \arrow[swap]{l}{(-) \circ (\Phi_{\msf{K}_r(\cat{X})} \circ J)^{\op}} 
					& \PSh_r(\Par(\cat{C}))
		\end{tikzcd}\end{center}
	By Proposition~\ref{equivRestr}, the right square commutes up to isomorphism. However, the left square also commutes up to 
	isomorphism as $\Phi_{\msf{K}_r(\cat{X})} \circ J$ is fully faithful. Hence the result follows.
\end{proof}

\begin{Cor}
	For any small restriction category $\cat{X}$, the embedding $\yon_r\colon \cat{X} \to\PSh_r(\cat{X})$ exhibits
	$\PSh_r(\cat{X})$ as the \emph{free restriction cocompletion} of $\cat{X}$.
\end{Cor}

\section{Free cocompletion of locally small restriction categories}\label{sec5}
So far in our discussions, we have considered the free cocompletion of a small $\M$-category $\cat{C}$ and of a small restriction
category $\cat{X}$, given by $\PSh_{\M}(\cat{C})$ and $\PSh_r(\cat{X})$ respectively. We now turn our attention to when
our categories may not necessarily be small, but locally small. In the case where $\cat{C}$ is an ordinary locally small category,
we understand the category of small presheaves on $\cat{C}$, denoted by $\P(\cat{C})$, to be its free cocompletion \cite{MR2324597}.
Recall that a presheaf on $\cat{C}$ is called small if it can be written as a small colimit of representables \cite{MR2324597}. 
We would like to first give a notion of free cocompletion of locally small $\M$-categories, and then give an analogue
in the restriction setting. To begin, we define what we mean by a locally small $\M$-category.

\begin{Defn}[Locally small $\M$-category]
	An $\M$-category $(\cat{C},\M_{\cat{C}})$ is called \emph{locally small} if $\cat{C}$ is locally small and 
	$\M$-well-powered. That is, for any object in $\cat{C}$, the $\M$-subobjects of $C$ form a small partially ordered set.
\end{Defn}

\begin{Rk}
	Note that this definition is exactly what is required for $\Par(\cat{C})$ to be a locally small restriction category when $\cat{C}$ is
	a locally small $\M$-category, as noted by Robinson and Rosolini \cite[p.~99]{MR968102}.
\end{Rk}

Now we know when $\cat{C}$ is an ordinary locally small category, $\P(\cat{C})$ is its free cocompletion. We also know that
for any small $\M$-category $(\cat{C},\M_{\cat{C}})$, its free cocompletion is given by 
$\PSh_{\M}(\cat{C}) = (\PSh(\cat{C}),\M_{\PSh(\cat{C})})$. This suggests that for any locally small $\M$-category
$\cat{C}$, we take $\P(\cat{C})$ as our base category and take its corresponding system of monics to be $\M_{\P(\cat{C})}$,
where $\M_{\P(\cat{C})}$ is defined in exactly the same way as for $\M_{\PSh(\cat{C})}$. Call this pair 
$(\P(\cat{C}), \M_{\P(\cat{C})}) = \P_{\M}(\cat{C})$. However, it is not immediately obvious that $\P_{\M}(\cat{C})$ is
an $\M$-category, since $\M_{\P(\cat{C})}$ may not be a stable system of monics. We therefore begin by showing that
$\M_{\P(\cat{C})}$ is stable.

\begin{Lemma}\label{PbInPPC}
	Let $\cat{C}$ be a locally small $\M$-category, and let $\mu\colon P \To Q$ be a map in $\M_{\P(\cat{C})}$. 
	If $\gamma\colon Q' \To Q$ is a map in $\PSh(\cat{C})$ with $Q'$ a small presheaf, then the pullback of $\mu$ along 
	$\gamma$ is in $\M_{\P(\cat{C})}$.
		\begin{center}\begin{tikzcd}
			P' \arrow{r} \arrow[swap]{d}{\mu'} & P \arrow{d}{\mu} \\
			Q' \arrow[swap]{r}{\gamma} & Q
		\end{tikzcd}\end{center}
\end{Lemma}

\begin{proof}
	Certainly $\mu'$ exists and is in $\M_{\PSh(\cat{C})}$ by the fact $\PSh_{\M}(\cat{C})$ is an $\M$-category. So all we need 
	to show is that $P'$ is a small presheaf. Since $Q'$ is small, we may rewrite $Q' \cong \colim \yon D$ for some functor 
	$D\colon \cat{I}\to\cat{C}$ with $\cat{I}$ small, and denote the colimiting coprojections as $q_I \colon \yon D_I \to Q'$. Now 
	$\mu$ is an $\M_{\P(\cat{C})}$-map, which means that for each $I\in\cat{I}$ and composite $\gamma\circ q_I$, there exists an
	$m_I\colon C_I \to D_I$ making the outer square a pullback.
		\begin{center}\begin{tikzcd}
			\yon C_I \arrow{r}{p_I} \arrow[swap]{d}{\yon m_I} & P' \arrow[swap]{d}{\mu'} \arrow{r} & P \arrow{d}{\mu} \\
			\yon D_I \arrow[swap]{r}{q_I} & Q' \arrow[swap]{r}{\gamma} & Q
		\end{tikzcd}\end{center}
	By the same argument as in the proof of Theorem \ref{MCatCocomp}, it follows that there is a functor $C \colon \cat{I}\to\cat{C}$ 
	which on objects, takes $I$ to $C_I$, and that there is a unique map $p_I \colon \yon C_I \to P'$ making the left square a pullback 
	for every $I\in\cat{I}$. However, because colimits are stable under pullback in $\PSh(\cat{C})$, this means 
	$(p_I\colon \yon C_I \to P')_{I\in\cat{I}}$ is colimiting, which ensures that $P'$ is a small presheaf.
\end{proof}

\begin{Rk}\label{PbInPPCRk}
	Note that the previous result implies that $\P(\cat{C})$ admits pullbacks along $\M_{\P(\cat{C})}$-maps, and that these
	are computed pointwise.
\end{Rk}

Having now shown that $\M_{\P(\cat{C})}$ is a stable system of monics, and hence $\P_{\M}(\cat{C})$ is an 
$\M$-category, we claim that $\P_{\M}(\cat{C})$ is indeed the free cocompletion of $\cat{C}$. To do so however, will first 
require showing that $\P_{\M}(\cat{C})$ is both locally small and cocomplete.

\begin{Lemma}
	If $\cat{C}$ is a locally small $\M$-category, then $\P_{\M}(\cat{C})$ is locally small.
\end{Lemma}

\begin{proof}
	It is well-known that $\P(\cat{C})$ is a locally small category \cite{MR2324597}, so all that remains is to show that
	$\P_{\M}(\cat{C})$ is $\M$-well-powered. So let $Q$ be a small presheaf, and rewrite $Q \cong \colim \yon D$, where 
	$D\colon \cat{I}\to\cat{C}$ is a functor with $\cat{I}$ small. Again denote the colimiting coprojections by
	$(q_I\colon \yon D_I \to Q)_{I\in\cat{I}}$.
	
	As before, if $\mu \colon P \To Q$ is an $\M$-subobject of $Q$, then there is a functor $C \colon \cat{I}\to\cat{C}$, 
	which on objects, takes $I \to C_I$, and unique maps $(p_I\colon \yon C_I \to P)_{I\in\cat{I}}$ making the following squares 
	pullbacks for each $I\in\cat{I}$:
		\begin{center}\begin{tikzcd}
			\yon C_I \arrow{r}{p_I} \arrow[swap]{d}{\yon m_I} & P \arrow{d}{\mu} \\
			\yon D_I \arrow[swap]{r}{q_I} & Q
		\end{tikzcd}\end{center}
	Note that $P\cong \colim \yon C$ as colimits are stable under pullback in $\P(\cat{C})$. There is also a natural 
	transformation $\alpha\colon C\To D$, given componentwise on $I$ by $m_I \in\M_{\cat{C}}$ 
	and whose naturality squares are pullbacks for every $I\in\cat{I}$. In fact, these functors from $\cat{I}$ to $\cat{C}$ 
	form an $\M$-category $(\cat{C}^{\cat{I}},\M_{\cat{C}^{\cat{I}}})$ whose $\M_{\cat{C}^{\cat{I}}}$-maps
	are just the natural transformations whose components are $\M_{\cat{C}}$-maps. Note that by observation,
	this $\M$-category $(\cat{C}^{\cat{I}},\M_{\cat{C}^{\cat{I}}})$ is locally small.
	
	It is not then difficult to see there is a function 
	$f \colon \Sub_{\M_{\P(\cat{C})}}(Q) \to \Sub_{\M_{\cat{C}^{\cat{I}}}}(D)$ taking the $\M$-subobjects of $Q$ to 
	the $\M$-subobjects of $D$. So to show that $\P_{\M}(\cat{C})$ is $\M$-well-powered, it is enough to show that $f$ is injective. 
	Let $\mu\colon P\To Q$ and $\mu\colon P'\To Q$ be two $\M$-subobjects of $Q$ which are mapped to the same 
	$\M$-subobject of $D$. That is, there is an isomorphism from $C$ to $C'$ making the following diagram commute:
		\begin{center}\begin{tikzcd}[row sep=small,column sep=small]
			C \arrow{rr}{\cong} \arrow[swap]{dr}{\alpha} & & C' \arrow{dl}{\alpha'} \\
			& D & 
		\end{tikzcd}\end{center}
	But because $P \cong \colim \yon C \cong \colim \yon C' \cong P'$, this induces an isomorphism between $P$ and $P'$ making
	the following diagram commute:
		\begin{center}\begin{tikzcd}[row sep=small,column sep=small]
			\yon C_I \arrow{rr}{p_I} \arrow{dr}{\cong} \arrow[swap]{dd}{\yon \alpha_I} & 
						& P \arrow[near end]{dd}{\mu} \arrow[dashed]{dr}{\cong} & \\
			& \yon C'_I \arrow{dl}{\yon \alpha'_I} \arrow[crossing over,near start]{rr}{p'_I} & & P' \arrow{dl}{\mu'} \\
			\yon D_I \arrow[swap]{rr}{q_I} & & Q &
		\end{tikzcd}\end{center}
	In other words, $\mu$ and $\mu'$ are the same $\M$-subobject of $Q$, and so the function $f$ is injective.
	Hence, if $\cat{C}$ is a locally small $\M$-category, then so is $\P_{\M}(\cat{C})$.
\end{proof}

Next, to show that $\P_{\M}(\cat{C})$ is cocomplete, we exploit Proposition \ref{MCocompAlt} and the following 
two lemmas.

\begin{Lemma}\label{PCCocomp1}
	Let $\cat{C}$ be a locally small $\M$-category. If $\{\mu_i \colon P_i \to Q_i \}_{i\in I}$ is a family of maps in 
	$\M_{\P(\cat{C})}$ indexed by a small set $I$, then so is their coproduct $\sum_{i\in I} \mu_i$.
\end{Lemma}

\begin{proof}
	Let $\{\mu_i \colon P_i \to Q_i \}_{i\in I}$ be a family of maps in $\M_{\P(\cat{C})}$, with $I$ some small set. To show 
	that $\sum_{i\in I} \mu_i$ is also in $\M_{\P(\cat{C})}$, we need to show that for any 
	$h\colon \yon D \to \sum_{i\in I} Q_i$, there is a map $m\colon C\to D$ in $\M_{\cat{C}}$ making
	the following diagram a pullback:
		\begin{center}\begin{tikzcd}
			\yon C \arrow{r} \arrow[swap]{d}{\yon m} & \sum_{i\in I}P_i \arrow{d}{\sum_{i\in I} \mu_i} \\
			\yon D \arrow[swap]{r}{h} & \sum_{i\in I} Q_i
		\end{tikzcd}\end{center}
	By Yoneda, there is a bijection 
		$$ \P(\cat{C})\left(\yon D,\sum_{i\in I} Q_i \right)\cong \left(\sum_{i\in I} Q_i \right)(D) \cong \sum_{i\in I} Q_i D, $$ 
	meaning that $h$ corresponds uniquely with some element $\tilde{h} \in \sum_{i\in I} Q_i D$. For each $i\in I$, the 
	naturality of the bijection $\P(\cat{C})(\yon D,Q_i)\cong Q_i D$ implies that $h\colon \yon D \to \sum_{i\in I} Q_i$ factors 
	through exactly one of the coproduct injections $\imath_{Q_j} \colon Q_j \to \sum_{i\in I} Q_i$. By extensivity of the 
	presheaf category $\PSh(\cat{C})$, the pullback of $\sum_{i\in I}\mu_i$ along $\imath_{Q_i}$ must be $\mu_i$. However, 
	as $\mu_j$ is an $\M_{\P(\cat{C})}$-map, there exists an $m\colon C\to D$ in $\M_{\cat{C}}$ making the left square of 
	the following diagram commute:
		\begin{center}\begin{tikzcd}
			\yon C \arrow{r} \arrow[swap]{d}{\yon m} & P_j \arrow{r}{\imath_{P_j}} \arrow[swap]{d}{\mu_j} 
					& \sum_{i\in I} P_i \arrow{d}{\sum_{i\in I} \mu_i} \\
			\yon D \arrow{r}{h'} \arrow[swap,bend right=20]{rr}{h} & Q_j \arrow{r}{\imath_{Q_j}} & \sum_{i\in I} Q_i
		\end{tikzcd}\end{center}
	Therefore, as both squares are pullbacks, $\yon m$ is a pullback of $\sum_{i\in I}\mu_i$ along $h$, which means 
	$\sum_{i\in I} \mu_i \in\M_{\P(\cat{C})}$.
\end{proof}

\begin{Lemma}\label{PCCocomp2}
	Let $\cat{C}$ be a locally small $\M$-category, and suppose $m$ is a map in $\P(\cat{C})$. If the pullback of $m$ along 
	some epimorphism is an $\M_{\P(\cat{C})}$-map, then $m$ must also be in $\M_{\P(\cat{C})}$.
\end{Lemma}

\begin{proof}
	Let $m\colon P\To Q$ be a map in $\P(\cat{C})$, and suppose $m'\colon P'\To Q'$ is a pullback of $m$ along some 
	epimorphism $f\colon Q'\To Q$. To show that $m$ is an $\M_{\P(\cat{C})}$, let $g\colon \yon D \To Q$ be any map
	in $\P(\cat{C})$. Again by Yoneda, there is a bijection $\P(\cat{C})(\yon D,Q)\cong QD$, giving a corresponding element 
	$\tilde{g} \in QD$. Since $f$ is an epimorphism in $\P(\cat{C})$, its component at $D$, $f_D\colon Q'D\to QD$, must also 
	be an epimorphism, which means there exists some element $\tilde{f'}\in Q'D$ such that $f_D(\tilde{f'}) = \tilde{g}$. The 
	naturality of the bijection $\P(\cat{C})(\yon D,Q)\cong QD$ then implies there is a map $f' \colon \yon D \to Q'$ such that
	$g = ff'$. Now using the fact $m'$ is an $\M_{\P(\cat{C})}$-map, there exists a map $n\in\M_{\cat{C}}$ such that
	$\yon n$ is the pullback of $m'$ along $f'$.
		\begin{center}\begin{tikzcd}
			\yon C \arrow{r} \arrow[swap]{d}{\yon n} & P' \arrow{r} \arrow[swap]{d}{m'} & P \arrow{d}{m} \\
			\yon D \arrow{r}{f'} \arrow[swap,bend right=20]{rr}{g} & Q' \arrow{r}{f} & Q
		\end{tikzcd}\end{center}
	Then as both squares are pullbacks, $\yon n$ must be the pullback of $m$ along $g=ff'$, making $m$ an 
	$\M_{\P(\cat{C})}$-map.
\end{proof}

We now prove that $\P_{\M}(\cat{C})$ is indeed cocomplete as an $\M$-category.

\begin{Lemma}
	Let $(\cat{C},\M_{\cat{C}})$ be a locally small $\M$-category. Then $(\P(\cat{C}),\M_{\P(\cat{C})})$ is a 
	cocomplete $\M$-category.
\end{Lemma}

\begin{proof}
	We begin by noting that the category of small presheaves on $\cat{C}$, $\P(\cat{C})$, is cocomplete. Therefore, it remains to
	show that the inclusion $\P(\cat{C}) \hookrightarrow \Par(\P(\cat{C}),\M_{\P(\cat{C})})$ is cocontinuous. However, 
	by Proposition \ref{MCocompAlt}, it is enough to show that the following conditions hold:
		\begin{enumerate}[label=(\alph*)]
			\item If $\{m_i \colon P_i \To Q_i\}_{i\in I}$ is a family of small $I$-indexed set of maps in $\M_{\P(\cat{C})}$, then 
						$\sum_{i\in I} m_i$ is also in $\M_{\P(\cat{C})}$ and the following squares are pullbacks for each $i\in I$:
						\begin{center}\begin{tikzcd}
							P_i \arrow{r}{\imath_{P_i}} \arrow[swap]{d}{m_i} & \sum_{i\in I} P_i \arrow{d}{\sum_{i\in I} m_i} \\
							Q_i \arrow[swap]{r}{\imath_{Q_i}} & \sum_{i\in I} Q_i
						\end{tikzcd}\end{center}
			\item Given the following diagram,
						\begin{center}\begin{tikzcd}
							P' \arrow[start anchor=20,end anchor=160]{r}{f'} 
										\arrow[start anchor=base east,end anchor=base west,swap]{r}{g'} \arrow[swap]{d}{m'} 
									& P \arrow{r}{c'} \arrow{d}{m} & G \arrow{d}{n} \\
							Q' \arrow[start anchor=20,end anchor=160]{r}{f} 
										\arrow[start anchor=base east,end anchor=base west,swap]{r}{g} & Q \arrow[swap]{r}{c} & H
						\end{tikzcd}\end{center}
					if $m\in\M_{\P(\cat{C})}$ and the left two squares are pullbacks, and $c,c'$ are the coequalisers of 
					$f,g$ and $f',g'$ respectively, then the unique map $n$ making the right square commute is in
					$\M_{\P(\cat{C})}$ and the right square is also a pullback.
			\item Colimits in $\P(\cat{C})$ are stable under pullback along $\M_{\P(\cat{C})}$-maps.
		\end{enumerate}
	To see that (c) holds, recall that $\P(\cat{C})$ admits pullbacks along $\M_{\P(\cat{C})}$-maps, and that these 
	are calculated pointwise as in $\Set$ (Remark \ref{PbInPPCRk}). The result then follows from the fact that colimits in 
	$\P(\cat{C})$ are also calculated pointwise together with the fact colimits are stable under pullback in $\Set$.
	
	For (b), it will be enough to show that the square on the right in (b) is a pullback (by Lemma \ref{PCCocomp2}). 
	Now the right square is a pullback in $\P(\cat{C})$ if and only if componentwise for every $A\in\cat{C}$, it is a pullback in 
	$\Set$. So consider the diagram in (b) componentwise at $A\in\cat{C}$: 
		\begin{center}\begin{tikzcd}
			P'A \arrow[start anchor=20,end anchor=160]{r}{f'_A} 
						\arrow[start anchor=base east,end anchor=base west,swap]{r}{g'_A} \arrow[swap]{d}{m'_A} 
					& PA \arrow{r}{c'_A} \arrow{d}{m_A} & GA \arrow{d}{n_A} \\
			Q'A \arrow[start anchor=20,end anchor=160]{r}{f_A} 
						\arrow[start anchor=base east,end anchor=base west,swap]{r}{g_A} & QA \arrow[swap]{r}{c_A} & HA
		\end{tikzcd}\end{center}
	The two left squares remain pullbacks in $\Set$, and $c_A,c'_A$ remain coequalisers of $f_A,g_A$ and $f'_A,g'_A$ 
	respectively since colimits in $\P(\cat{C})$ are calculated pointwise. Observe also that $m_A$ is a monomorphism
	as maps between small presheaves in $\P(\cat{C})$ are monic if and only if they are componentwise monic for every
	$A\in\cat{C}$ (by a Yoneda argument). Now we know that the $\M$-category $(\Set,\M_{\Set})$ (where $\M_{\Set}$ 
	are all the injective functions) is a cocomplete $\M$-category (Example \ref{SetM}), 
	and since $m_A$ is monic, the square on the right must be a pullback in $\Set$. Therefore, as pullbacks in $\P(\cat{C})$ 
	are calculated pointwise, the square on the right of (b) must  also be a pullback.
	
	For (a), we know that $\sum_{i\in I}m_I \in M_{\P(\cat{C})}$ from Lemma \ref{PCCocomp1}. Then, as
	$(\Set,\M_{\Set})$ is cocomplete and both pullbacks and colimits in $\P(\cat{C})$ are computed pointwise as in $\Set$,
	the result follows by an analogous argument to (b).
	
	Therefore, $(\P(\cat{C}),\M_{\P(\cat{C})})$ is a cocomplete $\M$-category.
\end{proof}

\begin{Thm}{(Free cocompletion of locally small $\M$-categories)}
	Let $\cat{C}$ be a locally small $\M$-category, and let $\cat{D}$ be a locally small, cocomplete $\M$-category. Then 
	the following is an equivalence of categories:
		$$ (-)\circ\yon \colon \M\cat{Cocomp}(\P_{\M}(\cat{C}),\cat{D}) \to \M\CAT(\cat{C},\cat{D}) $$
	where $\M\CAT$ is the $2$-category of locally small $\M$-categories.
\end{Thm}

\begin{proof}
	The proof follows exactly the same arguments presented in the proof of Theorem \ref{MCatCocomp}.
\end{proof}

\begin{Cor}{(Free cocompletion of locally small restriction categories)}
	For any locally small restriction category $\cat{X}$ and locally small, cocomplete restriction category $\E$, the following is 
	an equivalence of categories:
		$$ (-)\circ \eqref{CL} \colon \cat{rCocomp}(\Par(\P_{\M}(\M\Total(\msf{K}_r(\cat{X})))),\E) \to \cat{rCAT}(\cat{X},\E) $$
	where \eqref{CL} is the Cockett and Lack embedding and $\cat{rCAT}$ is the $2$-category of locally small
	restriction categories.
\end{Cor}

Given that a small presheaf on an ordinary category is one that can be written as a colimit of small representables, it is
natural to ask whether there is a similar notion of small restriction presheaf. So let $\cat{X}$ be a locally small restriction category,
and denoting the $\M$-category $\M\Total(\msf{K}_r(\cat{X}))$ by $\cat{C}$, the previous corollary says that 
$\Par(\P_{\M}(\cat{C}))$ is the free cocompletion of $\cat{X}$. Since $\P(\cat{C})$ is a full replete subcategory of 
$\PSh(\cat{C})$, and $\Par(\P_{\M}(\cat{C})) \simeq \PSh_r(\cat{X})$, there exists a full subcategory of $\PSh_r(\cat{X})$ 
which is equivalent to $\Par(\P_{\M}(\cat{C}))$:
	\begin{center}\begin{tikzcd}
		\P_r(\cat{X}) \arrow[hook]{d} \arrow{r}{\simeq} & \Par(\P_{\M}(\cat{C})) \arrow[hook]{d} \\
		\PSh_r(\cat{X}) \arrow[swap]{r}{\simeq} & \Par(\PSh_{\M}(\cat{C}))
	\end{tikzcd}\end{center}
where the above square is a pullback and the bottom map is the equivalence from Theorem \ref{CLembed}.

To see what objects should be in $\P_r(\cat{X})$, it is enough to apply $\Total$ to the above diagram, giving the following 
pullback:
	\begin{center}\begin{tikzcd}
		\Total(\P_r(\cat{X})) \arrow[hook]{d} \arrow{r} & \P(\Total(\msf{K}_r(\cat{X}))) \arrow[hook]{d} \\
		\Total(\PSh_r(\cat{X})) \arrow[swap]{r}{G} & \PSh(\Total(\msf{K}_r(\cat{X})))
	\end{tikzcd}\end{center}
where $G$ is an equivalence. Since the above diagram is a pullback, an object $P$ will be in $\Total(\P_r(\cat{X}))$
(and hence in $\P_r(\cat{X})$) if $GP$ is an object in $\P(\Total(\msf{K}_r(\cat{X})))$; that is, 
$GP \cong \colim \yon C_I$, where $C\colon \cat{I}\to\Total(\msf{K}_r(\cat{X}))$ is a functor with $\cat{I}$ small.
If we define $H$ to be a pseudo-inverse for $G$, then an object will be in $\P_r(\cat{X})$ if it is of the form
$P \cong \colim H\yon C_I $, for some small $\cat{I}$ and functor $C\colon \cat{I} \to \Total(\msf{K}_r(\cat{X}))$. 
We call these $P$ the \emph{small restriction presheaves}.

We also give an explicit description of a small restriction presheaf as follows. Since $GP$ is an object in 
$\P(\Total(\msf{K}_r(\cat{X})))$, it will be the colimit of a small diagram whose vertices are of the form $\yon(A,e)$, where 
$(A,e)$ is an object in $\msf{K}_r(\cat{X})$. Now given $(A,e)\in\msf{K}_r(\cat{X})$, note the following splitting
in $\PSh_r(\cat{X})$:
	\begin{center}\begin{tikzcd}[column sep=small]
		& Q(A,e) \arrow[tail]{dr} & \\
		\yon_r A \arrow[swap]{rr}{\yon_r e} \arrow[two heads]{ur} & & \yon_r A	
	\end{tikzcd}\end{center}
This gives a functor $Q\colon\msf{K}_r(\cat{X})\to\PSh_r(\cat{X})$. Then a restriction presheaf is called \emph{small} 
if it is the colimit of some functor $D\colon\cat{I}\to\PSh_r(\cat{X})$ ($\cat{I}$ small), where each $DI$ is of the form
$Q(A,e)$ for some $(A,e)\in\msf{K}_r(\cat{X})$, and each $D(f\colon I \to J)$ is total. We denote by $\P_r(\cat{X})$ 
the restriction category whose objects are small restriction presheaves on $\cat{X}$. By construction, it is also the 
free cocompletion of $\cat{X}$. It is not difficult to check that when $\cat{X}$ is a small restriction category, 
restriction presheaves on $\cat{X}$ are small, and so $\P_r(\cat{X}) = \PSh_r(\cat{X})$.

\bibliographystyle{plain}
\bibliography{References}

\end{document}